%% file: main.tex
\documentclass[runningheads]{llncs}
\usepackage[T1]{fontenc}
\usepackage{graphicx}
\usepackage{wrapfig}
\usepackage{pgf,tikz}
\usepackage{tikz-cd} % for commutative diagram
\usepackage{floatrow} % for caption on the side of images (-> persistence diagrams)
\usepackage{adjustbox}
\usepackage[ruled,noend]{algorithm2e} % ruled argument recommanded by lipics

\usetikzlibrary{arrows}

\graphicspath{{./}} %Setting the graphicspath
\makeatletter
\def\input@path{{./}}
\makeatother

\usepackage{amsmath}
\usepackage{amsthm}
\usepackage{amsfonts}
\usepackage{amssymb}
\usepackage{eucal}
\usepackage{pifont} % for check mark

\usepackage{enumitem}
\usepackage{xcolor}
\definecolor{mygreen}{rgb}{0.0, 0.6, 0.0}
\usepackage{hyperref}
\usepackage[capitalize]{cleveref} % for the cref commands
\usepackage{fancyvrb} % for changing font size in verbatim
\usepackage{blkarray} % for indices on matrices
\usepackage{multicol}
\usepackage{mathrsfs}%

\usepackage[title]{appendix}%
\usepackage[section]{placeins}%for figure in right section
\let\Oldsection\section
\renewcommand{\section}{\FloatBarrier\Oldsection}
\let\Oldsubsection\subsection
\renewcommand{\subsection}{\FloatBarrier\Oldsubsection}
\let\Oldsubsubsection\subsubsection
\renewcommand{\subsubsection}{\FloatBarrier\Oldsubsubsection}

\graphicspath{{./}} %Setting the graphicspath
\makeatletter
\def\input@path{{./}}
\makeatother

\makeatletter
\newcommand{\newreptheorem}[2]{
    \newtheorem*{rep@#1}{\rep@title}
    \newenvironment{rep#1}[1]{
        \def\rep@title{#2 \ref*{##1}}
        \begin{rep@#1}
    }{
        \end{rep@#1}
    }
}
\makeatother
\newenvironment{delayedproof}[1]
 {\begin{proof}[\raisedtarget{#1}\textbf{Proof of~\Cref{#1}}]}
 {\end{proof}}
\newcommand{\raisedtarget}[1]{%
  \raisebox{\fontcharht\font`P}[0pt][0pt]{\hypertarget{#1}{}}%
}
\newcommand{\proofref}[1]{\hyperlink{#1}{Proof of~\Cref{#1}}}
\usepackage{xparse}
\DeclareMathOperator{\im}{im}	% im
\DeclareMathOperator{\M}{\mathtt{M}}
\DeclareMathOperator{\W}{\mathtt{W}}

\DeclareMathOperator{\g}{\mathtt{g}}
\DeclareMathOperator{\f}{\mathtt{f}}

\newcommand\Span{\mathrm{Span}\,}
\newcommand{\dr}{\partial}
\newcommand{\inc}[2]{\mathtt{i}_{#1}}
\newcommand{\proj}[2]{\mathtt{j}_{#2}}
\newcommand{\cplx}{{(K,\partial)} }
\newcommand{\cplxp}{{(K',\partial')} }
\newcommand{\cplxpdr}{{(K',\partial)} }
\newcommand{\cplxarrow}[1]{\overset{\partial_{#1}}{\longrightarrow}}
\newcommand{\ols}[1]{\mskip.5\thinmuskip\overline{\mskip-.5\thinmuskip {#1} \mskip-.5\thinmuskip}\mskip.5\thinmuskip} % overline short

\newcommand{\explicit}{explicit}
\newcommand{\Explicit}{Explicit}
\newcommand{\HC}{HDVF}
\newcommand{\HCs}{HDVFs}
\newcommand{\HomC}{Homological Discrete Vector Field}

\newcommand{\Homcs}{Homological discrete vector fields}
\newcommand{\homc}{homological discrete vector field}

\newcommand{\mun}{^{-1}}

\newcommand{\bs}{\backslash}

\newcommand{\tx}[1]{\text{#1}}

\newcommand{\R}{\mathbb{R}}

\newcommand{\Z}{\mathbb{Z}}

\newcommand\id{id}

\newcommand{\equi}{\Longleftrightarrow}

\NewDocumentCommand\V{mggg}{%
	\ensuremath{\begin{pmatrix}#1\IfNoValueTF{#2}{}{\\#2\IfNoValueTF{#3}{}{\\#3\IfNoValueTF{#4}{}{\\#4}}}\end{pmatrix} }
}

\newcommand{\red}[1]{{\color{red}#1}}
\newcommand{\blue}[1]{{\color{blue}#1}}
\newcommand{\green}[1]{{\color{mygreen}#1}}
\newcommand{\primary}[1]{{\color{mygreen}\mathsf{#1}}}
\newcommand{\secondary}[1]{{\color{red}\mathsf{#1}}}
\newcommand{\critical}[1]{{\color{blue}\mathsf{#1}}}
\definecolor{DGreen}{RGB}{0, 105, 50}
\newcommand{\add}[1]{{#1}}%\color{DGreen}

\newcommand{\cheat}[1]{\vspace{-#1\textheight}} % \newcommand{\cheat}[1]{}

\NewDocumentCommand\twocolreport{mmgg}{%
	\begin{minipage}{\textwidth}
		\begin{minipage}[!t]{\IfNoValueTF{#3}{0.5}{#3}\textwidth}
			#1
		\end{minipage}
		\hfill
		\begin{minipage}[!t]{\IfNoValueTF{#4}{0.5}{#4}\textwidth}
			#2
		\end{minipage}%
	\end{minipage}
}

\usepackage{color}

\begin{document}

%Characterization of Combinatorially Computable (Co)Homology Bases
%Characterizing the (co)homology bases that we all compute

\title{Characterization of the computed homology and cohomology bases}

\author{{Yann-Situ Gazull}\inst{1} %\orcidID{009-003-8539-9275}
\and {Aldo Gonzalez-Lorenzo}\inst{1} %\orcidID{00-003-3226-7650} 
\and {Alexandra Bac}\inst{1} %\orcidID{009-005-3163-6882}
}
\institute{{Aix Marseille Univ, CNRS, LIS}, {{Marseille}, {France}}\\
\email{yann-situ.gazull@univ-amu.fr} \quad
\email{aldo.gonzalez-lorenzo@univ-amu.fr} \quad
\email{alexandra.bac@univ-amu.fr}
}

% %\titlerunning{Abbreviated paper title}
% % If the paper title is too long for the running head, you can set
% % an abbreviated paper title here
% %
% \author{First Author\inst{1}\orcidID{00-1111-2222-3333} \and
% %https://overleaf.g-mod.lis-lab.fr/project/652e973eeb03c9008d8e7602

% \input{reviews}

\maketitle              % typeset the header of the contribution

\begin{abstract}
  Computing homology and cohomology is at the heart of many recent works and a key issue for topological data analysis.
Among homological objects, homology generators are useful to locate or understand holes (especially for geometric objects).
\add{The present paper provides a characterization of the class of homology bases that are computed by standard algorithmic methods.}
% The present paper proves that standard algorithmic approaches compute \add{a specific class of} homology bases and gives a characterization of them. %(calling them \textit{explicit bases}).
The proof of this characterization relies on the Homological Discrete Vector Field, a combinatorial structure for computing homology, which encompasses several standard methods (persistent homology, tri-partitions, Smith Normal Form, discrete Morse theory).
These results refine the combinatorial homology theory and provide novel ideas to gain more control over the computation of homology generators.
\end{abstract}

%mandatory; please add comma-separated list of keywords

%\tableofcontents%to remove for submission

\section{Introduction}
% \commentAlex{
After having been a branch of mathematics from early 20th century, homology and cohomology have become a key issue in computer science since the 80's to study holes of data independently of geometry.
Later in the 90's, persistent homology even became a central tool for topological data analysis.
Indeed, it formalizes ``hole changes'' along a time-varying sequence of objects and computes a notion of ``hole importance'' or ``duration'', stable against noise.

(Co)homology computation can be carried out in several ways.
Algebraic computation, as a pseudo-diagonalization (Smith Normal Form) is widely used, for instance, for persistent homology.
On the other hand, effective homology provides a ``categorical'' approach (see~\cite{Sergeraert1986}), while Discrete Morse Theory proposes a combinatorial approach (see~\cite{forman-morse_theory_cell_complex}).
In 2017, Aldo-Gonzalez Lorenzo et al.~\cite{aldo-cycles_discrete_morse_theory} introduced Homological Discrete Vector Field (which we refer to as \emph{\HC} in the present paper), which %unifies state-of-art approaches and 
provides a framework \add{in between combinatorial approaches (such as discrete Morse theory) and algebraic approaches (such as Smith normal form), encompassing these state-of-art approaches but distinct from them}. \\ %In this setting, homology computations can be formalized as HDVF completions. \\
It has long been known that not all homology basis can be computed by standard algorithmic approaches (Smith Normal Form, persistent homology, tri-partitions, Discrete Morse Theory).
For instance, 1-dimensional computed generators are elementary cycles.
In the present paper, we introduce a precise characterization of homology bases computed by HDVFs, which we call \textit{explicit bases}.
We then show (or recall) the equivalence between specific HDVFs and other algorithmic methods and thus deduce that all of them actually compute such explicit bases.
%}

The paper is organised as follows:~\Cref{sect:preliminaries} introduces homology and \HCs.
\add{\Cref{sect:generators} contains our main contribution;} we define and discuss the notion of explicit homology bases, and prove that being \explicit\ is equivalent to being induced by a perfect \HC.
In~\Cref{sect:charac-consequences}, we explain how this characterization applies to homology bases computed by other standard methods.
In particular, we prove that tri-partitions (defined in~\cite{edelsbrunner-tri_partition}) can be seen as layers of perfect \HCs\ (and vice versa) and show that a homology basis is \explicit\ if and only if it is induced by a tri-partition.
We also explain how this characterization applies to persistent homology bases.

\section{Preliminaries}\label{sect:preliminaries}

\subsection{Homology, complexes and notations}
The present work considers homology over a field $\mathcal{F}$.
In the rest of the paper, we consider that $\mathcal F = \Z/2\Z$ but results hold for any field.
In this section, we quickly introduce the main concepts of homology theory.
%We also define the notations and conventions used in this paper.
For more information on homology theory, see~\cite{hatcher-algebraic_topology}.

%\REVIEW{Cellular homology is defined imprecisely in the preliminaries}
%\commentAlex{Je suis embêtée car la définition de l'opérateur de bord sur les CW-complexes ne me paraît pas évidente ... comme les recollements sont des applications de la sphère unité dans $ X^{n-1}$...} 
A cellular complex $\mathscr C$ is a topological structure that consists of cells of different dimensions ($\mathscr C_q$ the set of cells of dimension $q$) attached together by ``gluing'' maps (for a more precise definition, see~\cite{hatcher-algebraic_topology}).
Depending on the type of cells, cellular complexes are categorized into a variety of combinatorial structures: simplicial complexes and delta-complexes (triangles, tetrahedra, etc.), cubical complexes (squares, cubes, etc.), or general cellular complexes.
In this article, we mainly use cubical complexes, but the results apply to any cellular complex.

The definition of homology for these structures is based on the algebraic version of the concept of boundary.
The boundary morphism $\dr_q$ must associate to a $q$-cell a linear combination of $(q-1)$-cells corresponding to its geometric faces, while respecting the constraint that $\dr_{q} \circ \dr_{q+1} = 0$.
When such a definition is possible, we obtain a chain complex, an abstract algebraic structure on which we can compute homology:
\begin{definition}[chain complex~\cite{hatcher-algebraic_topology}]\label{def:chain-complex}
    A \emph{chain complex} $(K,\dr)$ is a sequence of $\mathcal{F}$-vector fields $K_q$ with linear maps $\dr_q: K_q \to K_{q-1}$ satisfying $\dr_{q} \circ \dr_{q+1} = 0$:
    $$
	K_n \cplxarrow{n} \dots \cplxarrow{q+1} K_q \cplxarrow{q} K_{q-1} \cplxarrow{q-1} \dots \cplxarrow{1} K_{0} \cplxarrow{0} 0
	$$
    $K_q$ is the space of \emph{$q$-chains} and $\dr_q$ is the \emph{$q$-boundary map}.
    When $(K,\dr)$ is derived from a cellular complex $\mathscr C$, the $q$-chains are defined as the linear combinations of $q$-cells: $K_q = \Span\mathscr C_q$.
    The boundary map $\dr$ is defined as the algebraization of the geometric boundary (see~\cite{hatcher-algebraic_topology}).
\end{definition}
The notion of \HC\ is deeply related to the relation between cells and chains.
That is why this paper only considers chain complexes that are derived from cellular complexes, which we simply refer to as \textit{complexes}.

Given a complex $(K,\dr)$, we denote $K := \bigoplus_{q=0}^n K_q$ the whole chain space and $\dr := \bigoplus_{q=0}^n \dr_q$ the boundary operator.
We also denote $\ols{K_q} := \mathscr{C}_q$ the set of $q$-cells and $\ols{K} := \bigcup_{q=0}^n \ols{K_q}$ the set of every cell in the associated cellular complex.
We respectively call \emph{$q$-cycles} and \emph{$q$-boundaries} the chains in $\ker\dr_q$ and in $\im\dr_{q+1}$.

%\commentAldo{Is it a good idea to use $K$ for the chain complex instead of the simplicial complex that generates it?}

%\commentAlex{\textbf{ET} : je serais pour avoir un nom spécifique pour les complexes (CW, simpliciaux) pour ne pas les confondre avec les complexes de chaîne. Je vous propose $K$ pour les complexes de chaînes puisque c'est ta notation, et $\mathscr{C}$ pour les complexes géométriques. Du coup $K_q := \Span{\mathscr C_q}$ me paraît plus clair ... ?}

%\commentYS{Je préfère en général le $K$ plutôt que le $\mathcal C$ car $\mathcal C$ ressemble trop à $C$ et ça m'avait déjà confus précédemment avec les HCs. Une alternative plus claire est $\mathscr C$. $\ols{K}$ est pas super comme notation non plus... mais je l'ai utilisé un peu partout}

% Complexe cellulaire ... phrase 1 OK
% Selon le type de cellules, les complexes cellaire se déclinent en une variété de structures combinatoires: complexes simpliciaux et delta-complexes (triangles, tétra...), complexes cubiques (carrés, cubes...) ou complexes cellulaires généraux. 
%Dans cet article nous travaillons principalement avec des complexes cubiques mais les résultats s'appliquent à tout complexe cellulaire.
% La définition d'une homologie pour ces structures repose sur l'algébrisation de la notion de bord. Le morphisme de bord doit associer à une q-cellule une combinaison linéaire de q-1 cellules correspondant à ses faces géométriques tout en respectant la contrainte que d^2 = 0.
% Quand une telle définition est possible, on obtient un complexe de chaîne: abstract algebraic ...

\begin{definition}[homology group~\cite{hatcher-algebraic_topology}]\label{def:homology-group}
    Given a complex $(K,\dr)$, its \emph{$q$-homology group} is defined as the quotient space of $q$-cycles over $q$-boundaries:
    $$H_q(K) = {\ker\dr_q}/{\im\dr_{q+1}}$$
    As $\mathcal{F}$ is a field, there exists an integer $\beta_q$ (called the $q$-th Betti number) such that $H_q(K)$ is isomorphic to $\mathcal{F}^{\beta_q}$.
    Intuitively, $\beta_q$ counts the number of holes of dimension $q$ in the associated cellular complex.
\end{definition}
We denote $[\cdot]^K_q : \ker\dr \to H_q(K)$ the map that associates a $q$-cycle to its corresponding $q$-homology class: $[x]^K_q = [y]^K_q \equi x-y \in \im\dr_{q+1}$.

\begin{definition}[homology bases and generators]\label{def:generators-homology-bases}
    A family $(\g_i)_{i\leq\beta}$ of $q$-cycles is a \emph{$q$-homology basis} of $\cplx$ if and only if $\left([\g_i]^K_q\right)_{i\leq\beta}$ is a basis of $H_q(K)$.
    The cycles $\g_i$ are called \emph{$q$-homology generators}.
\end{definition}

\paragraph{Notations and conventions:} %We give here some notations and conventions that will be used in the present paper.
given two complexes $\cplx$ and $\cplxp$, we will write $\cplx \subseteq \cplxp$ if $\cplx$ is a sub-complex of $\cplxp$ i.e. if $K\subseteq K'$ and $\dr'$ agrees with $\dr$ on $K$. % i.e $\dr = \proj{}{K}\dr'\inc{K}{}$; where $\inc{K}{B}$ is the inclusion map from $K$ and $\proj{B}{K}$ is the projection map to $K$.
For simplicity, we will often identify $\dr$ and $\dr'$. % and write $\cplx \subseteq \cplxpdr$.
If $A$ is a set of cells, we will write $\inc{A}{}$ the inclusion map from \add{the subspace $\Span A$ into the full chain space $K$, $\proj{}{A}$ the corresponding projection map from $K$ to $\Span A$} and $\id_{A}$ the identity map on $\Span A$.
Given two sets of cells $A$ and $B$, we also denote $\dr_{AB} := \proj{}{ B}\circ \dr \circ \inc{ A}{}$ the boundary induced by $\Span A$ restricted to $\Span B$.

\subsection{Reductions and \HCs}
Effective homology theory was introduced in the 90's in~\cite{Sergeraert1986}.
It defines the notion of \textit{reduction} between two chain complexes (usually from a large to a small complex), which ensures that their homology groups are isomorphic.
%\REVIEW{There is a definition of "reduction" from effective homology (see Definition 4). But the idea or the motivation behind this notion is not explained.}
Reductions play an important role in understanding the homology of complexes.
%They are related to \HC\ because every \HC\ induces a reduction.
\begin{definition}[reduction~\cite{Sergeraert1986}]\label{def:reduction}
    A \emph{reduction} from a complex $(K,\dr)$ to a complex $(K',d)$ is a triplet of maps $(f,g,h)$ with $f_q:K_q\to K'_q$,\quad $g_q:K'_q\to K_q$ and $h_q : K_q\to K_{q+1}$ satisfying the following properties:
    $$
        \mathbf{1.}~f \dr = d f \quad
        \mathbf{2.}~\dr g = g d \quad
        \mathbf{3.}~fh, hg, hh = 0 \quad
        \mathbf{4.}~f g = \id_{K'} \quad
        \mathbf{5.}~g f = id_K - \dr h - h \dr
    $$
    \centerline{\input{resources/tikz/reduction-diagram}} \\    
    Then we have $\forall q,\ H_q(K) \simeq H_q(K')$.
\end{definition}
%\commentAlex{J'ai dégagé la proposition 1 et intégré à la fin de la définition ...}

%\begin{proposition}[reduction preserves homology~\cite{Sergeraert1986}]\label{prop_reduction-homology}
%    If there exists a reduction $(f,g,h)$ from $(K,\dr)$ to $(K',d)$, then we have $\forall q,\ H_q(K) = H_q(K')$.
%\end{proposition}

\begin{definition}[perfect reduction~\cite{Sergeraert1986}]\label{def:perfect-reduction}
    A reduction from $(K,\dr)$ to $(K',d)$ is perfect iff $d=0$.
    In this case, the $q$-homology group of both complexes is isomorphic to $K'_q$.
    In addition, $(g_q(\gamma))_{\gamma \in \ols{K'_q}}$ is a $q$-homology basis.
\end{definition}

\Homcs\ emerged as a generalization of Discrete Morse Theory~\cite{forman-morse_theory_cell_complex}.
Discrete Morse theory uses acyclic graphs (called \textit{discrete gradient vector fields}) defined on the cells of the complex and provides bounds on Betti numbers. However, it was shown that it fails to compute homology, for instance, for well-known contractible but not collapsible complexes such as the Bing's house. %this approach cannot compute the homology of any cellular complex and provided counterexamples such as the Bing's house.
\HCs\ were first developed to bypass this limitation.

\begin{definition}[\homc~\cite{aldo-phd_thesis} (\HC)]\label{def:hdvf}\label{prop_hvdf-reduction}
  Given a complex $(K,\dr)$, a \emph{\homc} is a triplet $X=(P,S,C)$ with $P$, $S$, and $C$ forming a partition of $\ols K$ such that $\dr_{SP}$ is invertible.
  The cells of $P$, $S$ and $C$ are respectively called \emph{primary}, \emph{secondary} and \emph{critical} cells.\\
  \add{It is proved in~\cite{aldo-phd_thesis} that the invertibility of $\dr_{SP}$ implies that the following $(f,g,h)$ is a well defined reduction from $(K,\dr)$ to $(\Span C,d)$}:
    % \cheat{0.01}
    \begin{align*}
      g :=
      \begin{blockarray}{cc}
      C \\
      \begin{block}{(c)c}
      0 & ~P \\
      G & ~S \\
      id & ~C \\
      \end{block}
      \end{blockarray}
      &&f :=
      \begin{blockarray}{cccc}
      P & S & C \\
      \begin{block}{(ccc)c}
      F & 0 & id & ~C \\
      \end{block}
      \end{blockarray}
      &&h :=
      \begin{blockarray}{cccc}
      P & S & C \\
      \begin{block}{(ccc)c}
      0 & 0 & 0 & ~P \\
      H & 0 & 0 & ~S \\
      0 & 0 & 0 & ~C \\
      \end{block}
      \end{blockarray} %\\
      &&d :=
      \begin{blockarray}{cc}
      C \\
      \begin{block}{(c)c}
      D & ~C \\
      \end{block}
      \end{blockarray}
    \end{align*}
    % \cheat{0.02}
    
    with\quad $H=(\dr_{SP})\mun$,\quad $F = -\dr_{SC} \cdot H$,\quad $G = -H \cdot \dr_{CP}$\ \ and\ \ $D = \dr_{CC} - \dr_{SC} \cdot H \cdot \dr_{CP}$.
    \add{Given an \HC\ this reduction is called the \textit{associated reduction}, an example is detailed in the appendix (see~\Cref{example:reduction}).}
  
    In figures, primary, secondary and critical cells are respectively colourized in {\color{mygreen}green}, {\color{red}red} and {\color{blue}blue} (see~\Cref{fig:partitions,fig:hc-examples} for some examples).
\end{definition}

\add{Let us point out that \HCs\ are singular objects, in between algebra and combinatorics, %without being reducible to either of them,
interleaving the notions of cells and chains.
Indeed, \HCs\ are not mere reductions as they rely on a cell partition and not on a space decomposition; but they are not purely combinatorial objects either, as \HC\ validity condition is algebraic (matrix invertibility).} % 

\begin{definition}[perfect \HC~\cite{aldo-phd_thesis}]\label{def:perfect-hc}
    A \HC\ is perfect if its induced reduction is perfect, i.e. if the reduced boundary $d = 0$.
    In this case, $\Span C_q$ is isomorphic to $H_q(K)$ and $g$ maps every critical cell to its associated homology generator: $(g_q(\gamma))_{\gamma\in C_q}$ is the $q$-homology basis associated to $(P,S,C)$.
    Similarly, $f^\ast$ \add{(the transpose of $f$)} maps every critical cell to its associated cohomology generator: $(f_q^\ast(\gamma))_{\gamma\in C_q}$ is the $q$-cohomology basis associated to $(P,S,C)$. \add{See~\Cref{sect:explicit-cohomology-bases} for more details on duality and cohomology}, and see~\Cref{fig:hc-basis-examples} for an illustration.
\end{definition}

\begin{figure}
    \cheat{0.02}
    \centering
    \includegraphics[width = 0.24\textwidth]{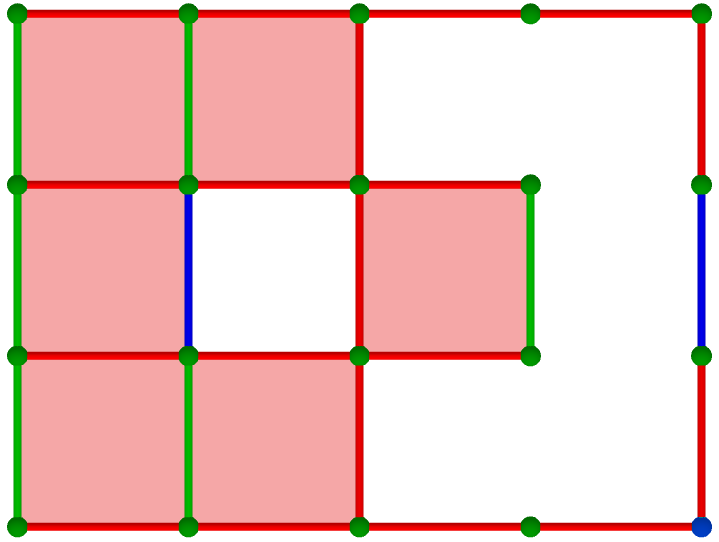} \hfill
    \includegraphics[width = 0.25\textwidth]{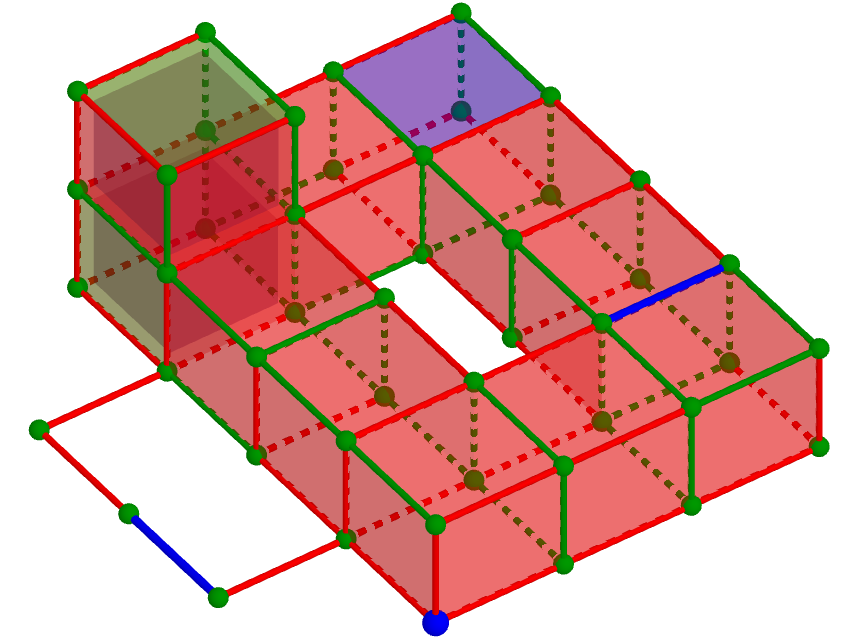}\hfill
    \includegraphics[width = 0.18\textwidth]{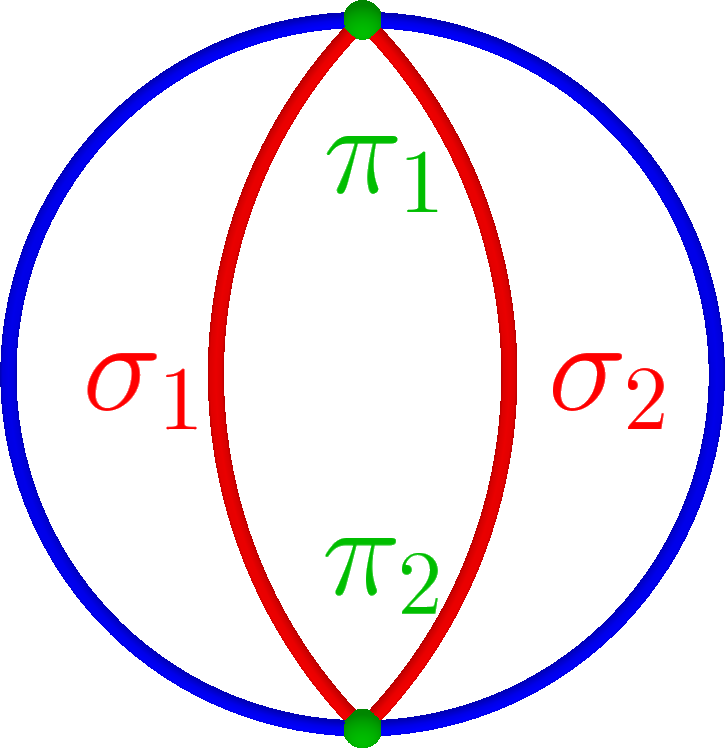}
    \caption{Examples of three complexes with partitions of the cells between primary (in {\color{mygreen}green}), secondary (in {\color{red}red}) and critical (in {\color{blue}blue}) cells.
    The first two partitions are \HCs\ whereas the last one is not a \HC: $\dr_{SP}$, the boundary of the secondary cells restricted to the primary cells, is not invertible because $\dr_{SP} (\sigma_1 + \sigma_2) = 0$.}
    \label{fig:partitions}
\end{figure}

\begin{figure}[!htb]
    \cheat{0.02}
    \centering
    \includegraphics[width = 0.27\textwidth]{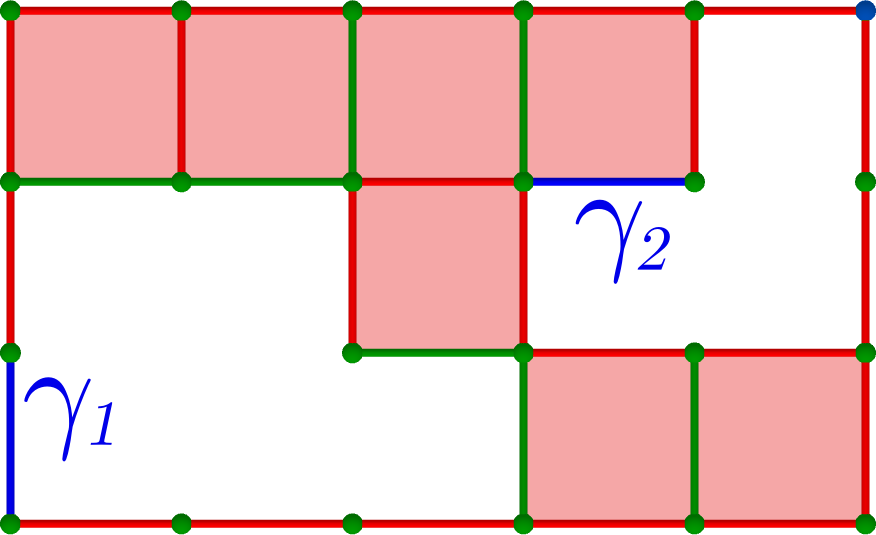}\hfill
    \includegraphics[width = 0.27\textwidth]{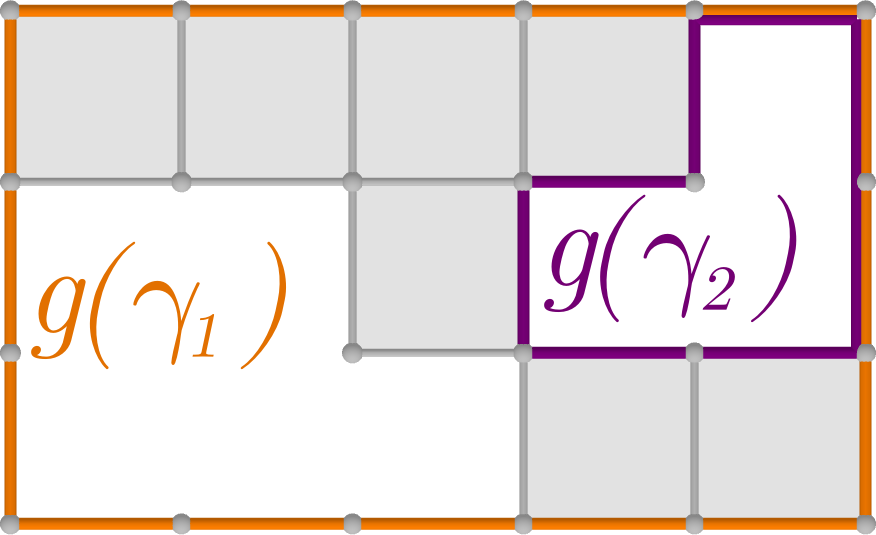}\hfill
    \includegraphics[width = 0.27\textwidth]{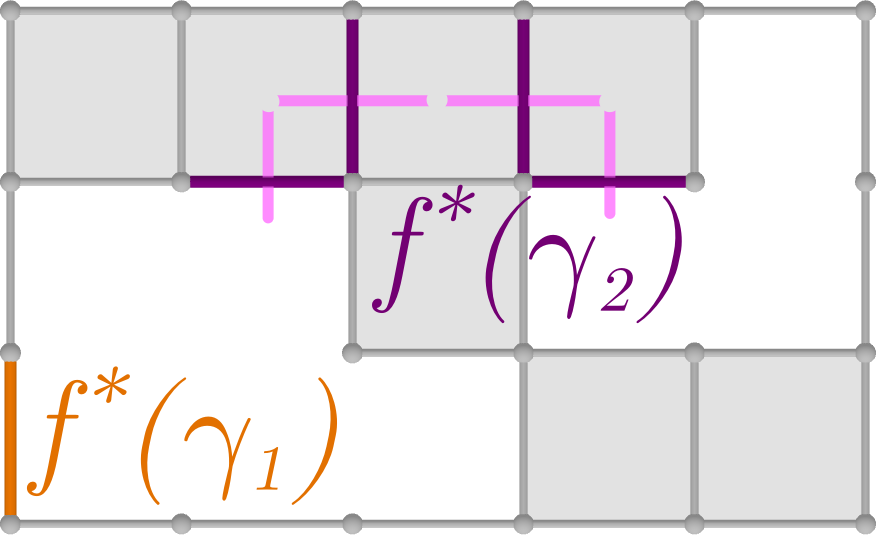}
    \caption{A perfect \HC\ $(P,S,C)$ with $C_1=\{\gamma_1, \gamma_2\}$ (left), its associated 1-homology basis $(g(\gamma_1), g(\gamma_2))$ (middle) and 1-cohomology basis $(f^\ast(\gamma_1),f^\ast(\gamma_2))$ (right).}
    \label{fig:hc-basis-examples}
\end{figure}
% \cheat{0.01}
%\commentAlex{Représenter dès ici la cohomologie par les cycles duaux => permet de tout mettre sur le même dessin à la figure 3}

Inspired by~\cite{edelsbrunner-tri_partition}, we define two maps which we call the \textit{canonical cycle} and \textit{co-cycle} maps.
Given a perfect \HC\ $(P,S,C)$ with reduction $(f,g,h)$, they are respectively defined as follows:
\begin{center}
    $ z_q = \id - h_{q-1}\dr_{q} \qquad z^q = \id - h_{q}^\ast\dr_{q+1}^\ast $
\end{center}
%The notation is the same as in~\cite{edelsbrunner-tri_partition}.
In particular, if $\gamma\in C_q$, the maps $z_q$ and $z^q$ coincide with the homology and cohomology generators: $ z_q(\gamma) = g_q(\gamma)$ and $z^q(\gamma) = f_q^\ast(\gamma)$.
This is a consequence of~\Cref{def:reduction}(5.) combined with the action of $f$, $g$ and $h$ on critical cells (see~\Cref{prop_hvdf-reduction}).
% Indeed, $\gamma$ is critical so $f(\gamma) = \gamma$ and $h(\gamma)=0$, so~\Cref{def:reduction}(5.) implies that $z_q(\gamma) = gf(\gamma) + \dr h(\gamma) = g(\gamma)$ (the proof is similar for $z^q(\gamma)$).
\add{Now let us state our first contribution; the following lemma will play a central role in what follows:}
\begin{lemma}\label{lem_canonical-cycle}
In a perfect \HC\ $(P,S,C)$:
\begin{center}
    $\forall x\in K_q \qquad (x+\Span S)\cap\ker\dr = \{z_q(x)\} \qquad (x+\Span P)\cap\ker\dr^\ast = \{z^q(x)\}$
\end{center}
  In particular, if $c\in\Span C_q$ we have $(c+\Span S)\cap\ker\dr = \{g_q(c)\}$.\\% and $(c+\Span P)\cap\ker\dr^\ast = \{f^\ast_q(c)\}$.
  In addition, we get $z_q \circ \inc{S_q}{} = 0$ and $z^q \circ \inc{P_q}{} = 0$.
\end{lemma}
\begin{proof}
The proof can be found in the appendix (see~\proofref{lem_canonical-cycle}).
\end{proof}

A core result of~\cite{aldo-phd_thesis} is the following: in a non perfect \HC, it is always possible to pair some critical cells (i.e. transform two critical cells to one primary and one secondary cell) in order to obtain a perfect \HC.
In particular, this proves the existence of a perfect \HC\ for every cellular complex.
This process is called \emph{the completion} of a \HC:
\begin{proposition}[completion of a \HC~\cite{aldo-phd_thesis}]\label{coro:hdvf-completion}
    Given two complexes $\cplx\subseteq \cplxp$ and a \HC\ $X=(P,S,C)$ for $\cplx$, there exists a perfect \HC\ $X'=(P',S',C')$ for $\cplxp$ such that $P\subseteq P'$ and $S \subseteq S'$.
\end{proposition}
% \begin{corollary}
%     Given two complexes $\cplx\subseteq \cplxp$ and a \HC\ $X=(P,S,C)$ for $\cplx$, there exists a perfect \HC\ $X'=(P',S',C')$ for $\cplxp$ such that $P\subseteq P'$ and $S \subseteq S'$.
% \end{corollary}
% \begin{proof}
% $X^\ast = (P,S, C\cap (\ols{K'}\bs \ols{K}))$ is a \HC\ for $\cplxp$ because $\dr'_{SP} = \dr_{SP}$ is invertible. Then, using~\Cref{prop_hdvf-completion} on $X^\ast$ we obtain $X'$ a perfect \HC\ for $\cplxp$ such that $P\subseteq P'$ and $S \subseteq S'$.
% \end{proof}
Three operations were defined in~\cite{aldo-phd_thesis} to locally modify a \HC\ (and hence homology an cohomology generators) by swapping the status of two set of cells.
We present two of them, that are used in our proofs:
\begin{definition}[\HC\ operations~\cite{aldo-phd_thesis}]\label{def:generalized-hdvf-operators}
  Given a \homc\ $X=(P,S,C)$ with induced reduction $(f,g,h)$ and $\Pi\subseteq P$,\ $\Sigma\subseteq S$ and $\Gamma\subseteq C$.
  The $\W$ operation swaps secondary cells with critical cells whereas the $\M$ operation swaps primary cells with critical cells:
  \begin{center}
      $\W_{\Sigma}^{\Gamma} X = \left(P,\ (S\cup\Gamma)\bs\Sigma,\ (C\cup\Sigma)\bs\Gamma \right) \qquad \M_{\Pi}^{\Gamma} X = \left( (P\cup\Gamma)\bs\Pi,\ S,\ (C\cup\Pi)\bs\Gamma \right)$
  \end{center} 
  $\W_{\Sigma}^{\Gamma} X$ is a \HC\ if and only if $\proj{}{ \Sigma} \circ g \circ \inc{ \Gamma}{}$ is invertible.
  $\M_{\Pi}^{\Gamma} X$ is a \HC\ if and only if $\proj{}{ \Gamma} \circ f \circ \inc{ \Pi}{}$ is invertible\footnote{\add{By a cardinality argument on critical cells, valid $\W$ and $\M$ operations transform a perfect \HC\ into another perfect \HC.}}. See~\Cref{fig:hc-examples} for illustrations.
  % The proof for the validity of these operations directly follows from that of~\Cref{prop_validity-hc-operators}, using matrix block products and identity matrices instead of matrix products and $1$s.
\end{definition}
Recall that these results only hold for homology over a field, and some limitations appear when we consider homology over a ring.

\begin{figure}
    \cheat{0.02}
    \centering
    \includegraphics[width = 0.95\textwidth]{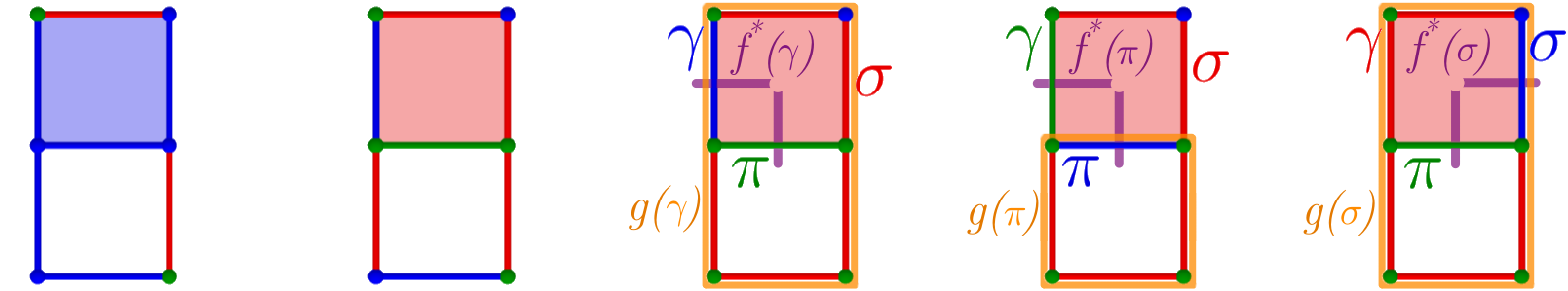}
    \caption{Five \HCs\ on the same complex: $X_1, \dots, X_5$. $X_1$ and $X_2$ are not perfect as the number of \blue{critical} $q$-cells is larger than the corresponding $q$-th Betti number. $X_2$ is a completion of $X_1$ (\add{three pairs of \red{secondary}-\green{primary} cells were added}) and $X_3$ is a completion of $X_2$. $X_3$, $X_4$ and $X_5$ are perfect. Using~\Cref{def:generalized-hdvf-operators} we can write $ X_4 = \M_{\{\pi\}}^{\{\gamma\}} X_3$ and $X_5 = \W_{\{\sigma\}}^{\{\gamma\}} X_3$.
    Note that $\W$ and $\M$ respectively preserves the homology generators (in orange) and cohomology generators (in purple) of the cells involved in the operation.}
    \label{fig:hc-examples}
\end{figure}

\cheat{0.03}
\section{\HC\ and Homology Bases}\label{sect:generators}
    Given a complex, a number of homology bases can be obtained using a perfect \HC, while others cannot (see~\Cref{fig:non-explicit-basis} for examples). \add{This section is our main contribution. We first define and  characterize bases induced by HDVFs (we call them ``explicit'' bases). In \Cref{thm_hdvf-generators-explicit} we show that the homology basis induced by a \HC\ is \explicit.
    Then, in \Cref{thm_explicit-generators-hdvf}, we derive the construction of a \HC\ from an \explicit\ homology basis.
    As a result, we prove that a homology basis is \explicit\ if and only if it is induced by a \HC.}
    %\commentAlex{A similar result can be obtained for cohomology bases, but we only present the homology one as it is more intuitive (see... for more details - lien sur le petit topo sur la cohomologie)}.}
    %In this section, we characterize bases induced by HDVFs (we call them ``explicit'' bases).
    
    Actually, this result goes beyond HDVFs and we show in section~\ref{sect:charac-consequences} that it extends to other standard approaches (such as persistent homology, Smith Normal Form or tri-partitions).
    %This is mainly a consequence of the combinatorial properties that a \HC\ must satisfy.
    %In this section, we investigate the properties of homology bases induced by \HCs.
    %First, we introduce the notion of \explicit\ homology basis.
    
    %More precisely, we then show that the homology basis induced by a \HC\ is \explicit\ (\Cref{thm_hdvf-generators-explicit}).
    %Finally, we derive the construction of a \HC\ from an \explicit\ homology basis (\Cref{thm_explicit-generators-hdvf}).
    %As a result, we prove that a homology basis is \explicit\ if and only if it is induced by a \HC.

\subsection{\Explicit\ homology bases}
    %We first introduce a class of homology bases, namely the class of 
    In this section we  introduce \textit{\explicit\ homology bases} and their properties.
    In what follows, $\cplx$ is a chain complex derived from a cellular complex $\mathscr C$.
    \add{The following results tightly intricate combinatorics (cells) and algebra (chains), which is consistent with the nature of \HCs.} 
    To avoid confusion between chains and cells, we use lowercase Greek letters for cells and lowercase Latin letters for chains.
    Cells are also considered as chains, therefore we identify a cell $\tau$ with the chain consisting of one $\tau$.
    Given a chain $x \in K$ we denote $\ols x$ the set of cells that appear in it: $\ols x := \left\{ \tau\in \ols K \ /\ \langle \tau, x \rangle \neq 0 \right\}$ (with $\langle \cdot, \cdot \rangle$ the standard inner product on $K$).
    As a consequence, if $A$ is a set of cells we have $a\in\Span A \equi \ols{a} \subseteq A$.
    
    \Cref{prop_explicit-bases-equivalence} gives various equivalent definitions of explicit bases and \Cref{prop_elementary-generators} shows that the generators of an \explicit\ basis are elementary.
    First, let us introduce the notion of sub-complex induced by a set of cells:

    \begin{definition}[induced complex]\label{def:induced-complex}
        Given a set of cells $A$ in a complex $\cplx$, we denote $(K(A),\dr)$ the minimum sub-chain complex containing $A$.
        Precisely, $K(A)_q$ is generated by the faces of dimension $q$ of cells in $A$.
    \end{definition}
    % The following properties follow directly from this definition: \commentYS{they are not used}
    % \begin{lemma}[properties of the induced complex]\label{lem_induced-complex-properties}
    %     ~
    %     \begin{enumerate}
    %         \item If $A \subseteq K_q$, $(K(A),\dr)$ is a complex of dimension $q$ ;
    %         \item $A \subseteq B \implies (K(A),\dr) \subseteq (K(B),\dr)$ ;
    %         \item $K(A\cup B) = K(A) \cup K(B)$
    %     \end{enumerate}
    % \end{lemma}

    Given a $q$-homology basis $(\g_i)_{i\leq\beta}$ for a complex $\cplx$ and $J \subseteq \{1, \dots, \beta\}$ a set of indices, we denote $K^J := K\left(\bigcup_{i \in J} \ols{\g_i}\right)$ the space of chains generated by the cells in the generators\footnote{\add{Note that $K^J_q$ is the span of a combinatorial object ($\cup_{i\in J} \ols{\g_i}$); it is not necessarily equal to $\Span\left((\g_i)_{i\in J}\right)$ (the ``natural'' algebraic object defined from generators).}} indexed by $J$.
    We denote $\iota_q^J$ the homology map induced by the inclusion $(K^J,\dr) \subseteq (K,\dr)$.
    With this notation, we also denote $(K^\beta,\dr) := (K^{\{1, \dots, \beta\}},\dr)$ the sub-complex induced by the whole basis (see~\Cref{fig:non-explicit-basis}) and $\iota_q^\beta$ its associated homology map.
    
    \begin{definition}[\Explicit\ homology bases]\label{def:explicit-bases}
        A $q$-homology basis $(\g_i)_{i\leq\beta}$ for a complex $\cplx$ is said to be \emph{\explicit} if it satisfies the following property:
        $$
            \forall k,\ \ols{\g_k}\bs\bigcup\nolimits_{i \neq k} \ols{\g_i} \neq \emptyset \qquad \tx{and}\qquad
            dim \left( H_q\left(K^{\beta}\right) \right) = \beta
        $$
        Intuitively, a homology basis is \explicit\ if and only if no generator is included in the others and the sub-complex induced by the generators have no more than $\beta$ holes of dimension $q$ (see~\Cref{fig:non-explicit-basis} for examples of non \explicit\ homology bases).
    \end{definition}
    
    \begin{figure}
        % \cheat{0.02}
        \centering
        \begin{tabular}{ccccc}
           \includegraphics[width = 0.17\textwidth]{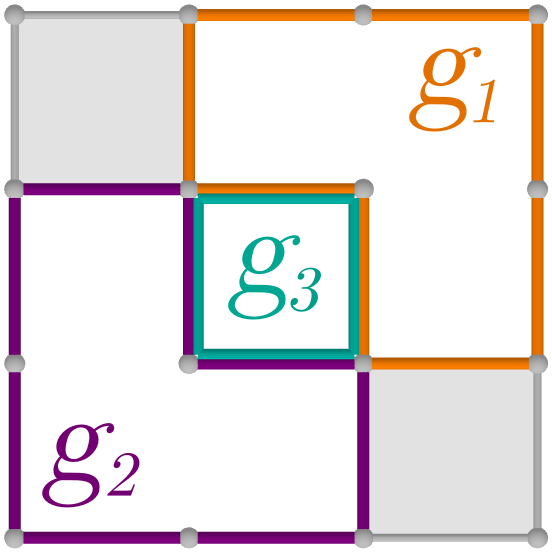}  &  
           \hspace{0.05\textwidth} &
           \includegraphics[width = 0.12\textwidth]{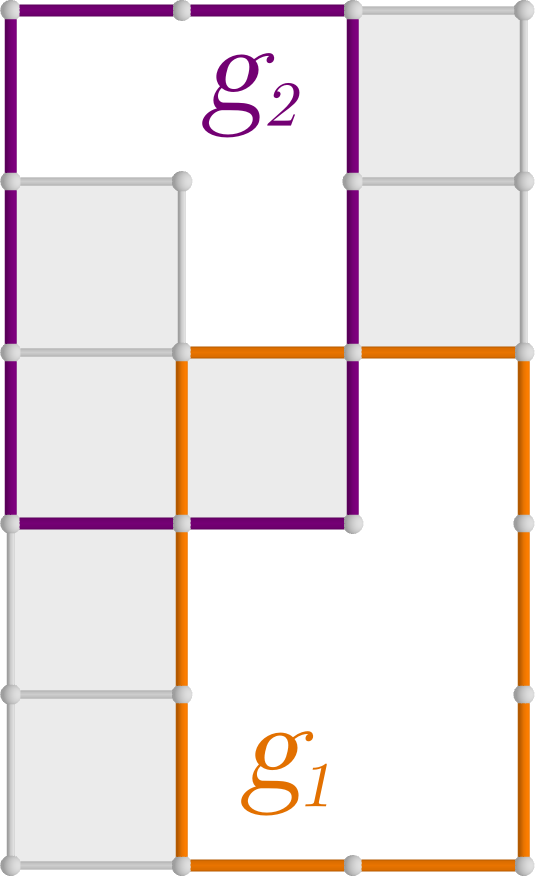} \hspace{0.005\textwidth} 
           \includegraphics[width = 0.12\textwidth]{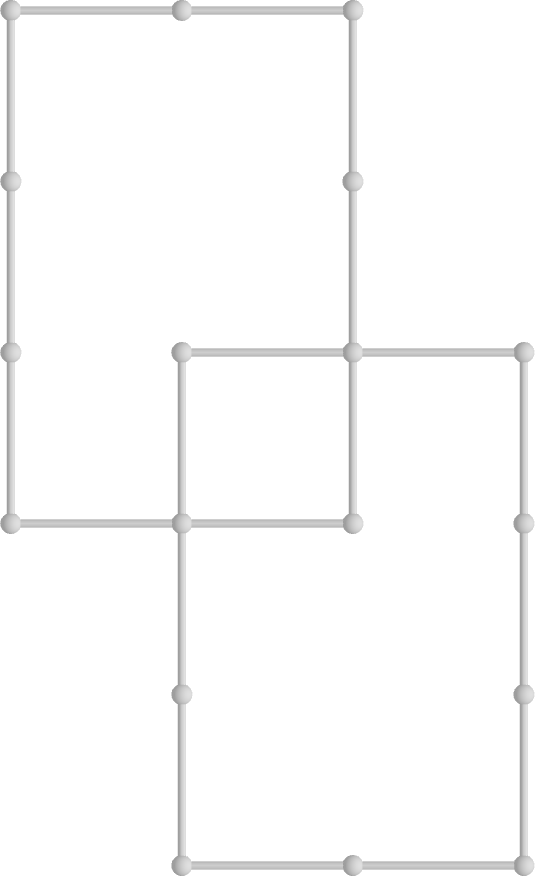} &
           \hspace{0.05\textwidth} &
           \includegraphics[width = 0.18\textwidth]{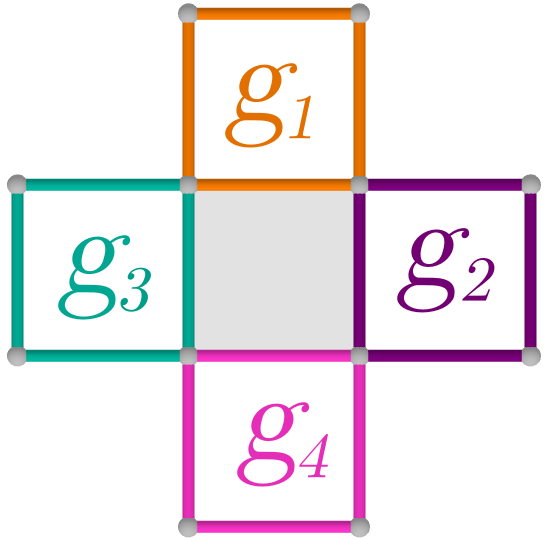} \hspace{0.005\textwidth}  
           \includegraphics[width = 0.18\textwidth]{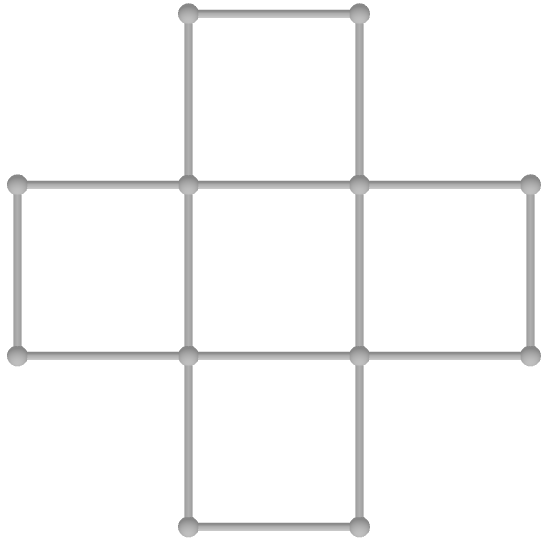} \\
           $(K_1)$ & & $(K_2)$ & & $(K_3)$ \\
        \end{tabular}
         
        \caption{Examples of $1$-homology bases which are not \explicit\ (right images: sub-complexes induced by the basis). By~\Cref{thm_hdvf-generators-explicit} they cannot be obtained from a \HC.\\
        %In the first image, 
        In $(K_1)$, $\ols{\g_3} \bs (\ols{\g_1} \cup \ols{\g_2}) = \emptyset$\ so the basis is not \explicit.
        %In the second example, 
        In $(K_2)$, the complex $(K^\beta,\dr)$ induced by the basis has three 1-holes instead of two: $3 = dim \left( H_1\left(K^{\beta}\right) \right) \neq \beta = 2$.
        Similarly, in $(K_3)$, the complex induced by the basis has five 1-holes instead of four.}
        \label{fig:non-explicit-basis}
    \end{figure}
    %\cheat{0.02}
    
    \begin{proposition}[\Explicit\ bases equivalent definitions]\label{prop_explicit-bases-equivalence}
        Given a $q$-homology basis $(\g_i)_{i\leq\beta}$ for a complex $\cplx$, the three following properties are equivalent:
        \begin{align*}
        1.\quad & \forall k,\ \ols{\g_k}\bs\bigcup\nolimits_{i \neq k} \ols{\g_i} \neq \emptyset \quad \tx{and}\quad
         \forall J,\ \iota_q^J: H_q\left(K^J\right) \to H_q(K) \ \tx{injective of rank $|J|$}\\
        2.\quad & \forall J\subseteq \{1, \dots, \beta\}, \quad
        \ker\dr \cap K^J_q \subseteq \Span\big((\g_i)_{i\in J}\big)\\
        3.\quad & \forall k,\ \ols{\g_k}\bs\bigcup\nolimits_{i \neq k} \ols{\g_i} \neq \emptyset \quad \tx{and}\quad
        dim \left( H_q\left(K^{\beta}\right) \right) = \beta
        \end{align*}
        %$$
        %\forall k,\ \ols{\g_k}\bs\bigcup_{i \neq k} \ols{\g_i} \neq \emptyset \qquad %\tx{and}\qquad
        %\beta_q\left(K^k\right) = k
        %$$
    \end{proposition}
    \begin{proof}
        The proof can be found in the appendix (see~\proofref{prop_explicit-bases-equivalence}).
    \end{proof}

    \add{Interestingly, all these definitions combine combinatorics (cells) and algebra (chains), but cannot be fully captured by either. Even 2. of~\Cref{prop_explicit-bases-equivalence} entails an intersection with $K^J_q$, which definition relies on cells.
    This dual nature of bases computed by standard approaches is an important result of the present article.}
    Moreover, we have the following property:
    \begin{proposition}\label{prop_elementary-generators}
        The generators of an \explicit\ homology basis are \emph{elementary}.
        A cycle is said to be \emph{elementary} if it is not composed of smaller (non-zero) cycles.
    \end{proposition}
    \begin{proof}
        \input{resources/proofs/prop_elementary-generators}
    \end{proof}
    
    %\commentAldo{I think that it is simpler to state the results only for homology, and say at the end that the same is true for cohomology} \commentYS{ good idea}\commentAlex{I agree :)}

\subsection{Characterization of \HC\ homology bases}
    We now prove that the homology bases obtained from \HCs\ are \explicit.
    \begin{theorem}\label{thm_hdvf-generators-explicit}
        Let $X=(P,S,C)$ be a perfect \HC\ for $\cplx$ with reduction $(f,g,h)$.
        Then, for all $q$, its associated $q$-homology basis $\big(g_q(\gamma)\big)_{\gamma\in C_q}$ is \explicit.
    \end{theorem}
    \begin{proof}
        This proof uses the characterization 2. of~\Cref{prop_explicit-bases-equivalence} and can be found in the appendix (see~\proofref{thm_hdvf-generators-explicit}).
    \end{proof}

    As stated by the following theorem, the converse also holds: %is also true:
    \begin{theorem}\label{thm_explicit-generators-hdvf}
    Given $(\g_i)_{i\leq\beta_q}$ an \explicit\ $q$-homology basis for $\cplx$, there exists a perfect \HC\ $X=(P,S,C)$ for $\cplx$ with reduction $(f,g,h)$ such that $\big(g_q(\gamma)\big)_{\gamma\in C_q} = (\g_i)_{i\leq\beta_q}$.
    \end{theorem}
    We prove~\Cref{thm_explicit-generators-hdvf} by building a perfect \HC\ from an \explicit\ basis.
    To do so, we need the following lemmas about existence and completion of \HCs.
    Firstly, given $\cplx\subseteq\cplxpdr$, it is possible to complete a perfect \HC\ for $\cplx$ to a perfect \HC\ for $\cplxpdr$ while preserving some critical cells (specifically the ones associated with holes of $\cplx$ that were not filled in $\cplxpdr$):
    % \begin{lemma}[critical cell preservation throughout completion] \label{lem_injective-hdvf-completion}
    %     Let $(K,\dr)\subseteq (K',\dr)$ be two complexes.
    %     Let $X=(P,S,C)$ be a perfect \HC\ for $\cplx$ with reduction $(f,g,h)$.\\
    %     Denote $\iota_q : H_q(K) \to H_q(K')$ be the $q$-homology morphism induced by the inclusion $K\subseteq K'$ and $[g]^K_q = [\cdot]_q^K \circ g_q : \Span C_q \to H_q(K)$ the isomorphism that associates a $q$-critical cell to its corresponding $q$-homology class.
    
    %     Suppose that $C^\ast \subseteq C$ is a subset of critical cells such that $\forall q,\ \iota_q \circ [g]^K_q \circ \inc{C^\ast_q}{C_q} : \Span C^\ast_q \to H_q(K')$ is injective.
    %     Then there exists $X'=(P',S',C')$ a perfect \HC\ for $\cplxpdr$ such that $P\subseteq P'$, $S\subseteq S'$ and $C^\ast \subseteq C'$.
    % \end{lemma}
    \begin{lemma}[critical preservation during completion] \label{lem_injective-hdvf-completion}
        Let $(K,\dr)\subseteq (K',\dr)$ and let $X=(P,S,C)$ be a perfect \HC\ for $\cplx$ with reduction $(f,g,h)$.
        
        Suppose that there is $C^\ast \subseteq C_q$ such that $( [g_q(\gamma)]^{K'} )_{\gamma \in C^\ast}$ is linearly independent in $H_q(K')$.
        Then there exists $X'=(P',S',C')$ a perfect \HC\ for $\cplxpdr$ such that $P\subseteq P'$, $S\subseteq S'$ and $C^\ast \subseteq C'$.
    \end{lemma}
    \begin{proof}
        The proof is technical and can be found in the appendix (see~\proofref{lem_injective-hdvf-completion}).
    \end{proof}
    
    Secondly, the critical cells preserved in a completion process keep their associated homology generators:
    \begin{lemma}[generator preservation during completion]\label{lem_generator-preservation}
        Suppose $\cplx\subseteq (K',\dr)$ and let $X=(P,S,C)$ and $X'=(P',S',C')$ be perfect \HCs\ for $\cplx$ and $(K',\dr)$ such that and $S \subseteq S'$.
        Then $\forall x \in K_q$, $z_q(x) = z_q'(x)$.\\
        In particular, $\forall \gamma \in C\cap C'$, $g(\gamma) = g'(\gamma)$.
    \end{lemma}
    \begin{proof}
        Let $x \in K \subseteq K'$.
        $X$ and $X'$ are perfect so~\Cref{lem_canonical-cycle} gives:
        
        \smallskip
        \qquad$ z_q(x) \in \ker\dr \cap (x + \Span S) \subseteq \ker\dr \cap (x + \Span S') = \{z'_q(x)\}$
    \end{proof}
    
    Finally, given an \explicit\ homology basis, it is possible to build a perfect \HC\ $X^\beta$ for the sub-complex $\left(K^\beta, \dr\right)$ induced by the basis, such that $X^\beta$ is still associated to the same basis:
    \begin{lemma}[\HC\ of the basis complex]\label{lem_explicit-generator-complex}
        Let $(\g_i)_{i\leq\beta}$ be an \explicit\ $q$-homology basis for $\cplx$.
        Then there exists $X^\beta=(P^\beta,S^\beta,C^\beta)$ a perfect \HC\ for $\left(K^\beta, \dr\right)$ with reduction $(f^\beta,g^\beta,h^\beta)$ such that $\big(g^\beta_q(\gamma)\big)_{\gamma\in C^\beta_q} = (\g_i)_{i\leq\beta}$.
    \end{lemma}
    \begin{proof}
        The construction of this \HC\ is done by induction on $k \in \{0, \dots \beta\}$ and can be found in the appendix (see~\proofref{lem_explicit-generator-complex}).
    \end{proof}
    
    %\commentYS{\Cref{lem_generator-preservation} is also useful to justify the definition of persistent generators.} 
    %The detailed proof of~\Cref{thm_explicit-generators-hdvf} can be found in the appendix (see~\proofref{thm_explicit-generators-hdvf}).
    The full proof of~\Cref{thm_explicit-generators-hdvf} is technical and can be found in appendix (see~\proofref{thm_explicit-generators-hdvf}). However, once previous lemmas are known, it becomes intuitive and~\Cref{fig:hc-explicit-example-1} illustrates it.

    As a result,~\Cref{thm_hdvf-generators-explicit,thm_explicit-generators-hdvf} imply that a homology basis is \explicit\ if and only if it is induced by a perfect \HC.
    Note that the proof of~\Cref{thm_explicit-generators-hdvf} provides an algorithm to build a perfect \HC\ given an explicit homology basis.
    
    % In addition, it is also possible to use duality to define the notion of \explicit\ cohomology bases.
    % The aforementioned theorems can then be translated in terms of cohomology, which gives a geometrical and topological characterization of cohomology bases obtained by \HCs.

    \begin{figure}
        \centering
        \begin{tabular}{ccccccc}
        \includegraphics[width = 0.17\textwidth]{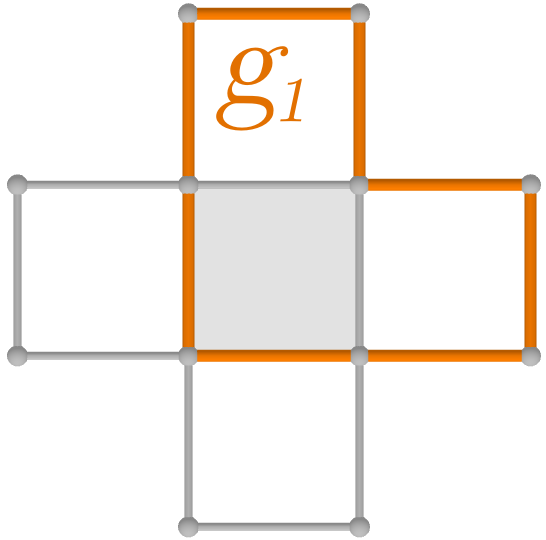} &
        \hspace{.05\textwidth} &
        \includegraphics[width = 0.17\textwidth]{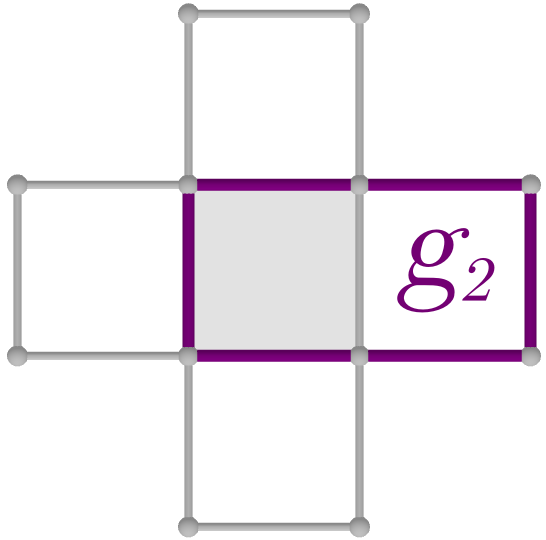} &
        \hspace{.05\textwidth} &
        \includegraphics[width = 0.17\textwidth]{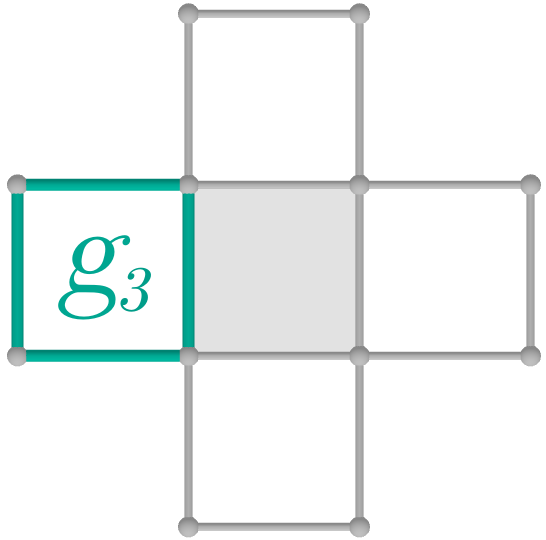} &
        \hspace{.05\textwidth} &
        \includegraphics[width = 0.17\textwidth]{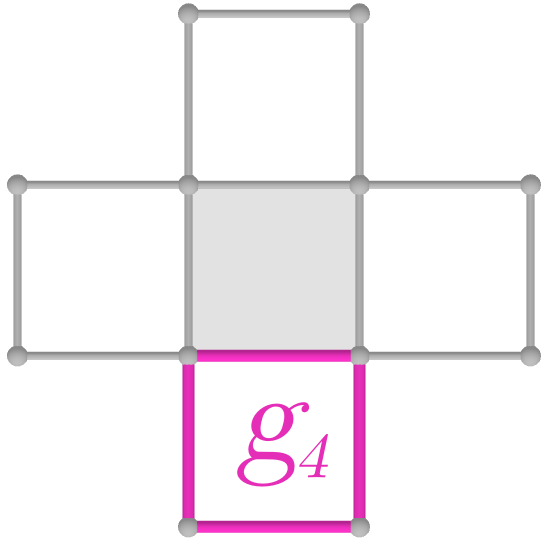}\\
        \vspace{0.02\textwidth}
        \includegraphics[width = 0.17\textwidth]{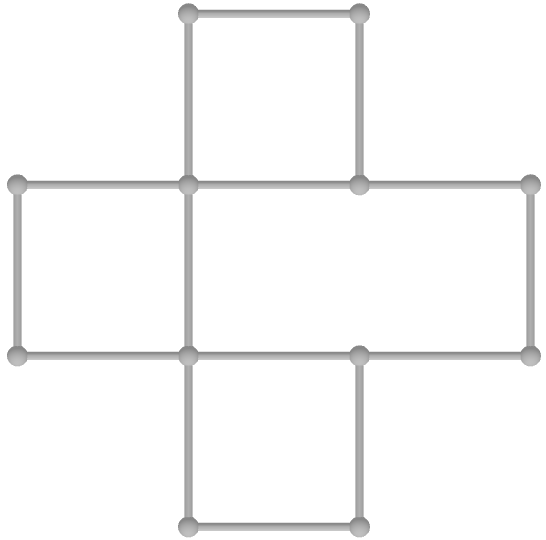} & &
        \includegraphics[width = 0.17\textwidth]{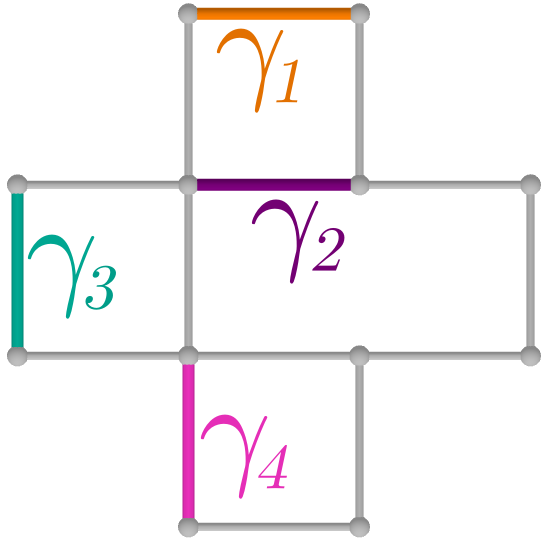} & &
        \includegraphics[width = 0.17\textwidth]{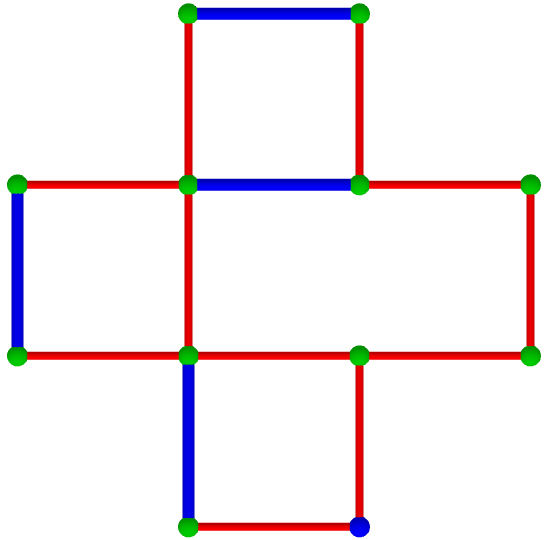}& &
        \includegraphics[width = 0.17\textwidth]{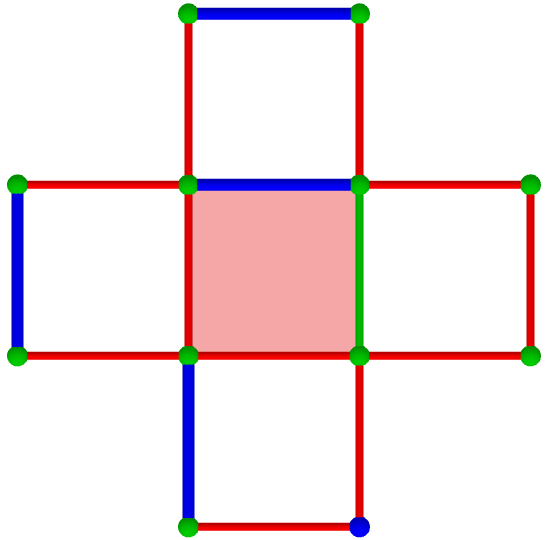} \\
        (a) & & (b) & & (c) & & (d) 
        \end{tabular}
        \caption{Illustration of the constructions involved in the proof of~\Cref{thm_explicit-generators-hdvf}.\\
        The first row represents four generators forming an \explicit\ $1$-homology basis of a complex $(K,\dr)$.
        The second row shows steps of the proof. It starts (a) with $(K(\bigcup \ols{\g_i}),\dr)$ the complex induced by this basis (which has four 1-holes as the basis is \explicit).
        (b) \add{represents a choice of cells} $\gamma_k$ that are respectively in $\ols{\g_k}\bs\bigcup_{i\neq k} \ols{\g_i}$ (which is non-empty as the basis is explicit).
        (c) is the perfect \HC\ $X^\beta$ obtained with~\Cref{lem_explicit-generator-complex}.
        By construction, we get $g^\beta(\gamma_i) = \g_i$.
        Finally, (d) is $X$ the perfect \HC\ for $(K,\dr)$ obtained by completing $X^\beta$ using~\Cref{lem_injective-hdvf-completion}.\\
        The $1$-homology basis induced by $X$ is $(\g_1, \g_2, \g_3, \g_4)$ by construction using~\Cref{lem_generator-preservation}.}
        \label{fig:hc-explicit-example-1}
    \end{figure}
% \cheat{0.01}
\section{Characterization of Homology Bases Obtained by Other Methods}\label{sect:charac-consequences}
    
    %\commentYS{Not enough space to talk about all that?}
    In this section we discuss how the characterization of~\Cref{sect:generators} transfers to bases computed by other methods.

    First, it is important to distinguish computational homology approaches (such as \HCs, but also Smith Normal Form, persistent homology and tri-partitions) and heuristics for computing small homology generators (the question of computing minimal ones is NP-difficult, see for instance~\cite{dey-minimal_cycles_hard_cases}).
    Methods related to this later question usually combine former approaches together with solvers and heuristics to minimize generators, hence homology bases may not be \explicit\ (see for instance~\Cref{fig:non-explicit-basis}, $K_3$).    
    Now, the remainder of this section examines the links between HDVFs and other approaches and deduces properties on the bases they compute.
    \smallskip
    
    \noindent \textit{Discrete Morse theory.}\quad It has been proven in~\cite{aldo-cycles_discrete_morse_theory} that \HCs\ generalize Discrete Morse Theory, in particular a \textit{discrete gradient vector field} is a \HC. Therefore, \Cref{thm_hdvf-generators-explicit} naturally applies, and  Discrete Morse Theory computes only explicit bases.
    \smallskip
    
    \noindent \textit{Smith Normal Form (SNF).}\quad Homology is commonly computed by reducing the boundary matrix of a complex to the Smith-Normal-Form (SNF).
    It has been shown in~\cite{aldo-cycles_discrete_morse_theory} that a \HC\ can emulate this SNF computation.
    Precisely, pairing secondary and primary cells is equivalent to performing a boundary matrix reduction step.
    Some works, such as~\cite{peltier-computation_groups_generators}, deduce homology bases from the SNF of the boundary matrix (or a modified version of it).
    Therefore, \Cref{thm_hdvf-generators-explicit} applies, and such methods (when homology is computed over a field) actually compute \explicit\ bases.
    %\commentYS{Maybe I need to add some justification?}\commentAlex{Je pense que c'est OK comme ça}
    
\subsection{Characterization of homology bases computed by tri-partitions}\label{sect:hc-tri-partition}

%\paragraph{Tri-partitions:}
    In a recent paper~\cite{edelsbrunner-tri_partition}, the authors defined the notion of tri-partition.
    A tri-partition is a combinatorial object that provides homological information using the notion of tree and cotree: % (in~\cite{skraba-minimal_spanning_acycles}, $q$-trees are referred to as acycles):
    \begin{definition}[trees and cotrees]\label{def:co-tree}
        A \emph{$q$-tree} is a set of $q$-cells $A$ that does not generate non-zero cycles, i.e $\Span A \cap \ker\dr = \{0\}$.
        Similarly, a \emph{$q$-cotree} is a set of $q$-cells $A$ that does not generate non-zero cocycles, i.e $\Span A \cap \ker\dr^\ast = \{0\}$.
        A $q$-tree $A$ (resp. $q$-cotree) is maximal if there is no other $q$-tree (resp. $q$-cotree) $A'$ such that $A\subset A'$.
    \end{definition}
    %\commentYS{In~\cite{skraba-minimal_spanning_acycles} they define $q$-trees as ``spanning acycles'', and have already some results such as: the size of a maximal spanning $q$-acycle is the rank of the boundary (basically $|S_q| = dim \im\dr_{q+1}$) or the set of spanning acycles is a matroid.}
    
    \begin{definition}[tri-partition~\cite{edelsbrunner-tri_partition}]\label{def:tri-partition}
        Given a complex $(K,\dr)$, a \emph{tri-partition} of dimension $q$ is a triplet $(A^q,A_q,E_q)$ forming a partition of the $q$-cells of $K$ such that $A^q$ is a maximal $q$-cotree, $A_q$ is a maximal $q$-tree and $|E_q|=\beta_q$.
    \end{definition}
    
    % \begin{theorem}[Theorem 3.1 in~\cite{edelsbrunner-tri_partition}]\label{thm_tri-partition}
    %     Let $K$ be a complex.
    %     Then there exist tri-partitions for every dimension $q$.
    % \end{theorem}
    It has been proved that a complex always admits tri-partitions for every dimension $q$ (see~Theorem 3.1 in~\cite{edelsbrunner-tri_partition}).
    Tri-partitions also provide homology and cohomology bases.
    \begin{definition}[canonical cycle~\cite{edelsbrunner-tri_partition}]\label{def:tri-partition-canonical-cycle}
        Let $(A^q,A_q,E_q)$ be a tri-partition for a complex $\cplx$.
        If $\tau \in A^q \sqcup E_q$, there is a unique $q$-cycle $z_q(\tau)$ in $\tau+\Span A_q$.
        By duality, if $\tau \in A_q \sqcup E_q$, there is a unique $q$-cocycle $z^q(\tau)$ in $\tau+\Span A^q$.
        %$z_q(\tau)$ and $z^q(\tau)$ are called the \emph{canonical} cycle and cocycle of $\tau$ respectively.
        %This defines two linear operators $z_q:K_q\to\ker\dr_q$ and $z^q:K_q\to\ker\dr_{q+1}^\ast$.
    \end{definition}
    It was proved that $(z_q(\epsilon) )_{\epsilon \in E_q}$ is a $q$-homology basis and $(z^q(\epsilon))_{\epsilon \in E_q}$ is a $q$-cohomology basis (see~Theorem 4.3 in~\cite{edelsbrunner-tri_partition}).
    \Cref{prop_hc-equiv-tri-partition} states that tri-partitions are closely related to perfect \HCs.
    Precisely, they can be viewed as ``dimensional layers'' of perfect \HC.
    Conversely, a perfect \HC\ can be viewed as a ``stack'' of tri-partitions.
    \begin{proposition}\label{prop_hc-equiv-tri-partition}
        If $(P,S,C)$ is a perfect \HC\ for $K$, then $(P_q, S_q, C_q)$ is a $q$-tri-partition for every dimension $q$.
        Conversely, if for every $q$, $(A^q,A_q,E_q)$ is a $q$-tri-partition of $K$, then $(\bigcup A^q, \bigcup A_q, \bigcup E_q)$ is a perfect \HC\ for $K$.
    \end{proposition}
    \begin{proof}
        The proof can be found in the appendix (see~\proofref{prop_hc-equiv-tri-partition}).
    \end{proof}
    
    This proposition underlines the fact that perfect \HCs\ and tri-partitions represent the same object, but the former uses a ``vertical'' point of view whereas the latter uses a ``horizontal'' point of view.
    As a consequence, \Cref{def:tri-partition-canonical-cycle} coincides with the canonical (co)cycle maps defined for \HC. This is a consequence of~\Cref{prop_hc-equiv-tri-partition} and~\Cref{lem_canonical-cycle}.
    Finally, the characterization from~\Cref{sect:generators} and the link between tri-partitions and \HC\ implies the following result:
    \begin{theorem}\label{thm_charac-bases-from-tri-partition}
        A $q$-homology basis is \explicit\ iff it is induced by a $q$-tri-partition.
    \end{theorem}

\subsection{Characterization of persistent homology bases}\label{sect:persistent-homology}
    Persistent homology studies  changes in the homology of a filtration.
    Formally, a filtration is a growing sequence of complexes $K^0 \subseteq K^1  \subseteq \dots \subseteq K^N $.
    The inclusions between complexes induces maps between the homology groups associated with the complexes.
    For each dimension $q$ we have:
    \begin{center}
        $H_q(K^0) \overset{\iota_q^0}{\longrightarrow} H_q(K^1) \overset{\iota_q^1}{\longrightarrow} \dots \overset{\iota_q^{N-1}}{\longrightarrow} H_q(K^N)$
    \end{center}
    %$$H_q(K^0) \overset{\iota_q^0}{\longrightarrow} H_q(K^1) \overset{\iota_q^1}{\longrightarrow} \dots 
    Given $i<j$, the image of $\iota_q^i \circ \dots \circ \iota_q^{j-1}$ represents the $q$-homology classes of $K^i$ that remain in the filtration until $K^j$.
    Intuitively, it represents the $q$-holes of $K^i$ that were not filled between $K^i$ and $K^j$.
    It is then possible to define a \textit{birth} and a \textit{death} index for each homology class that appears in the filtration.
    %Precisely, a $q$-homology class is born at index $i$ if it belongs to $H_q(K^i)\bs \im(\iota_q^{i-1})$.
    %Such a homology class dies at index $j$ if it belongs to $H_q(K^{j-1})\cap\ker(\iota_q^{j-1})$.
    Thus, every homology class in the filtration is associated with a \textit{lifetime} interval $[i,j]$ ($j$ eventually equals to $\infty$ if the homology class never dies).
    These information can be summed up in a \textit{persistence diagram}, which can be defined as a multiset over $\R^2$.
    See~\cite{edelsbrunner-persistent_homology,zomordian-computing_persistent_homology} for a detailed explanation of persistent homology.

    \medskip
    
%\paragraph{Characterization of homology bases:}\ 
    % Citation of Holes and dependences in an ordered complex:
    % However, given a monotonic ordering of the cells, we can use matrix reduction to construct a unique tri-partition. The cells in are exactly the cells in the monotonic ordering that give birth to essential homology classes; those in and are the cells that give birth and death to non-essential homology classes, see (Edelsbrunner and Ölsböck, 2018).
    %, because \HCs\ can emulate the computation of the SNF of the boundary matrix.
    % By tweaking the algorithm that computes a perfect \HC -- mainly to choose primary-secondary pairs the same way positive-negative pairs are chosen in~\cite{edelsbrunner-persistence_simplification} -- it is possible to compute the persistence diagram of a filtration.
    It has been shown in~\cite{aldo-cycles_discrete_morse_theory} that persistent homology can be computed by tweaking the algorithm that computes a perfect \HC. 
    Given a filtration $\emptyset = K^0 \subseteq \dots \subseteq K^N = K$, the method starts with an empty \HC\ and completes it into a perfect \HC\ $X^i = (P^i, S^i, C^i)$ for $(K^i,\dr)$ at the $i$-th iteration.
    Precisely, when a cell $\tau_i$ is added to the filtration at step $i$, it is paired (as a secondary-primary pair) with the youngest possible cell $\tau_j$ already present.
    If no such $\tau_j$ exists, $\tau_i$ is left as a critical cell.
    This algorithm actually reproduces how positive-negative pairs are chosen in~\cite{edelsbrunner-persistence_simplification}.
    Its complexity is the same as for computing a perfect \HC\ -- at most cubic in the number of cells.

    This method not only computes persistence intervals, but also \HCs\ and reductions at each step of the filtration.
    In particular, %it allows to track the homology and cohomology generators along the filtration:
    it computes a $q$-homology basis for each $(K^i,\dr)$, which is called the \textit{$i$-th persistent $q$-homology basis} and whose elements are called \textit{persistent $q$-homology generators}.
    % this allows to track the homology and cohomology generators along the filtration.
    % it computes a $q$-homology basis for each step of the filtration, which is called the \textit{$i$-th persistent $q$-homology basis} and whose elements are called \textit{persistent $q$-homology generators}.
    Formally, if $(f^i,g^i,h^i)$ is the reduction associated to $X^i$, the $i$-th persistent $q$-homology basis is given by $\bigl(g^i_q(\gamma)\bigl)_{\gamma \in C^i_q}$.
    Note that it is a \textit{persistent basis} in the sense of~\cite{kdey-persistent_1_cycles,wu-optimal_cycles_cardiac,obayashi-tightest_representative_persistent}.

    \add{A direct consequence of~\Cref{lem_generator-preservation} is that persistent generators are preserved during the lifetime of the homology class they represent.
    Equivalently, for all $\gamma \in C^i_q$, we have $z_q(\gamma) = g^i_q(\gamma)$, where $z_q$ is the canonical cycle map of the final \HC\ $X^N$.}
    % The following proposition implies that persistent generators are preserved during the lifetime of the homology class they represent.
    % Moreover, they can directly be read on the final \HC\ $X^N$:
    % \begin{proposition}\label{prop_persistent-generator-preservation}
    %     Let $z_q$ be the canonical cycle map of the final \HC\ $X^N$.
    %     Then for all $i$ and for all $\gamma \in C^i_q$, we have $z_q(\gamma) = g^i_q(\gamma)$.
    % \end{proposition}
    % This proposition is a direct consequence of~\Cref{lem_generator-preservation}.
    \add{Note that in~\cite{gonzalez-diaz-AT_models_persistent}, the authors prove similar results and define a similar ``tweaking'' of the algorithm computing \textit{AT-models} (i.e. perfect reductions) to emulate persistent homology.}

    As a result, every point $(i,j)$ in the persistence diagram can be associated to a unique cell $\gamma$ (which was critical between step $i$ and $j$) and to a unique persistent homology generator $z_q(\gamma)$.
    Theorem~\ref{thm_hdvf-generators-explicit} implies that the persistent homology bases of a filtration are \explicit.
    An illustration of persistent homology bases computed by \HC\ is given in~\Cref{fig:persistent-bases}.
    Note that other works such as~\cite{pham-representatives_persistent_double_twist,cufar-fast_persistent_representatives_involuted} deduce persistent homology bases from boundary matrix reductions.
    Again, as boundary matrix reductions can be reproduced by a \HC\ completion, the persistent homology bases computed by these methods are \explicit.

    \begin{figure}
    % \cheat{0.02}
        \centering
        \includegraphics[width = 0.92\textwidth]{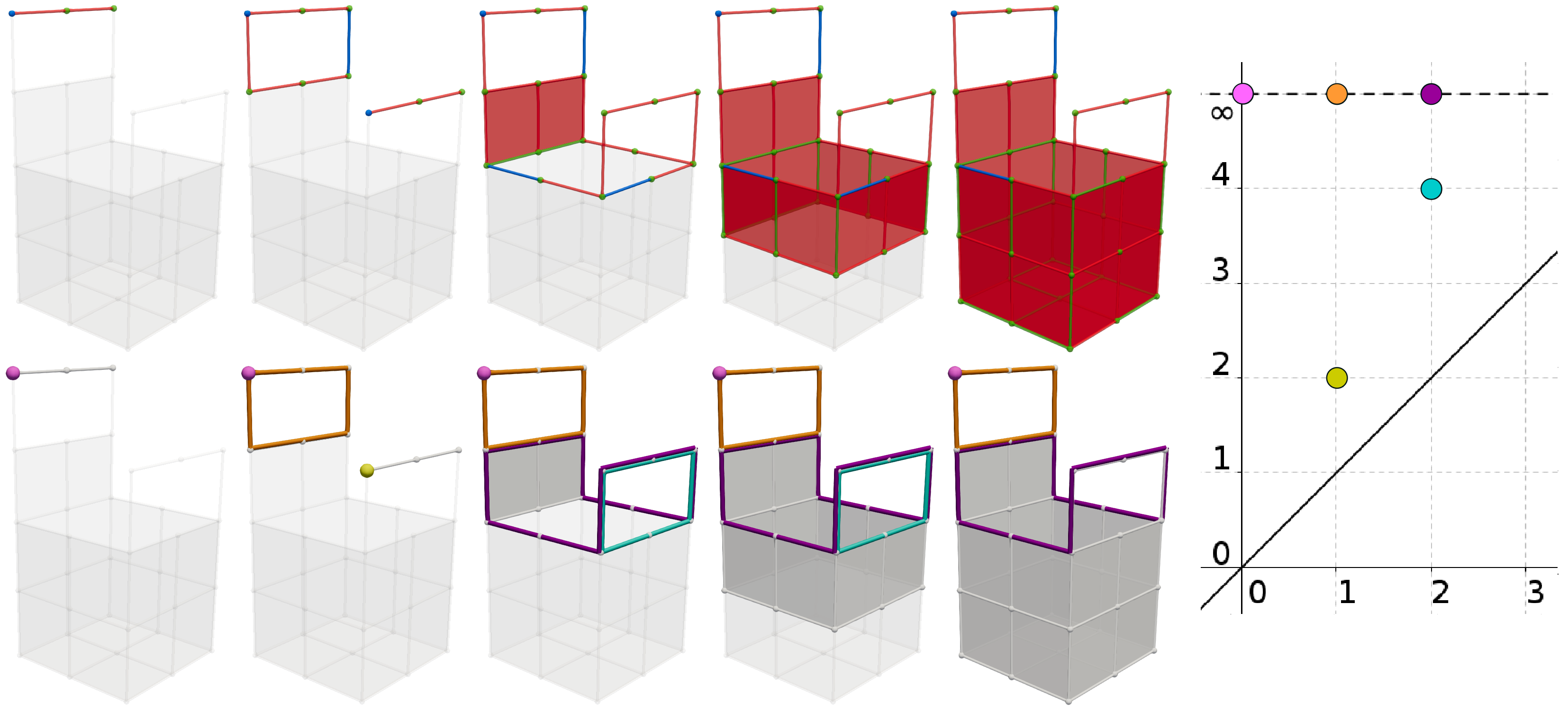}
        \caption{Example of persistent homology computed by \HC.
        First row: the filtration $K_0 \subseteq \dots \subseteq K_4$ and its \HCs\ $X^0, \dots, X^4$.
        Second row: the $i$th-persistent homology bases for $i = 0, \dots, 4$ (bases of dimension 0 and 1 are represented).
        On the right: the persistence diagram of the filtration, each point $(i,j)$ corresponds to a homology class whose lifetime is $[i,j]$, its color matches the color of the corresponding persistent generator.}
        \label{fig:persistent-bases}
    \end{figure}
    % \cheat{0.01}
    
\section{Conclusion and Future Works}
    
    %\REVIEW{These results are of interest, but one might regret the lack of real applications reaching outside of the framework of \HC.}
    This paper introduced a characterization of the homology bases computed by standard computational homology methods.
    We first proved that a class of homology bases called \textit{explicit} is equivalent to being induced by a \HC.
    We then showed how this result applies to other methods, in particular to tri-partitions (which we proved are closely related to perfect \HCs), Discrete Morse Theory and boundary matrix reduction methods (and thus standard persistent homology algorithms).
    Note that a similar result (yet less intuitive) holds for cohomology bases (see~\Cref{sect:explicit-cohomology-bases} in the appendix).

    This work suggests that the \HC\ framework provides a way to study computational homology methods.
    In particular, these results might give some insights on the isomorphisms computed in~\cite{aldo-constructive_alexander_duality} and on the hole closing and opening operations defined in~\cite{edelsbrunner-holes_dependences}.
    In addition, they also raise computational homology questions.
    For instance, can the set of \explicit\ bases over a complex give information about the topology? 
    Can we use \HC\ operations to adapt a perfect \HC\ to a given \explicit\ basis?
    Are minimal \explicit\ bases easier to compute than standard minimal homology bases?
    
    Another open area is the study of \HCs\ when homology is computed over a ring.
    In this case, the existence of a perfect reduction is not guaranteed due to the presence of torsion. However, it might be possible to define a reduction that is somehow ``optimal'' and investigate the \HCs\ related to it.

\subsubsection{\discintname}
On behalf of all authors, the corresponding author declares that there is no conflict of
interest.
% \begin{credits}
% % \subsubsection{\ackname}
% % Blablabla blablabli

% \subsubsection{\discintname}
% On behalf of all authors, the corresponding author declares that there is no conflict of
% interest.
% \end{credits}
% \bibliographystyle{plain}
\bibliographystyle{splncs04} % -> Package natbib Error: Bibliography not compatible with author-year citations.
\bibliography{references}

%%%%%%%%%%%%%%%%%%%%%%%%%%%%%%%%%%%%%%%%%%%%%%%%%%%%%%%%%%%%%%%%%%%%
\section*{Appendix}
% \red{\textit{As discussed with DGMM organisers, the appendix is provided here but will be put in an extended version on arXiv if the submission is accepted.}}
%\\input\{resources/proofs/(.*?)\}
% See the \proofref{lem_epsilon-close-charac} in Appendix.
\subsection{Example and Proofs of~\Cref{sect:preliminaries}}

\begin{example}[reduction of a \HC]\label{example:reduction}
    \input{resources/reduction-example}    
\end{example}

\begin{delayedproof}{lem_canonical-cycle}
    \input{resources/proofs/lem_canonical-cycle}
\end{delayedproof}

% \begin{delayedproof}{prop_structure-co-cycles}
%     \input{resources/proofs/prop_structure-co-cycles}
% \end{delayedproof}
% \begin{delayedproof}{prop_structure-co-boundaries}
%     \input{resources/proofs/prop_structure-co-boundaries}
% \end{delayedproof}

\subsection{Proofs of~\Cref{sect:generators}}

\begin{delayedproof}{prop_explicit-bases-equivalence}
    \input{resources/proofs/prop_explicit-bases-equivalence}
\end{delayedproof}

% \begin{delayedproof}{prop_structure-generators}
%     \input{resources/proofs/prop_structure-generators}
% \end{delayedproof}

\begin{delayedproof}{thm_hdvf-generators-explicit}
    In this proof, we denote $P^\bot := S \sqcup C$.
     \input{resources/proofs/thm_hdvf-generators-explicit}
\end{delayedproof}

\begin{delayedproof}{lem_injective-hdvf-completion}
    \add{
    \input{resources/proofs/lem_injective-hdvf-completion}
    }
\end{delayedproof}

\begin{delayedproof}{lem_explicit-generator-complex}
    \input{resources/proofs/lem_explicit-generator-complex}
\end{delayedproof}

\begin{delayedproof}{thm_explicit-generators-hdvf}
    \input{resources/proofs/thm_explicit-generators-hdvf}
\end{delayedproof}

\subsection{Proofs of~\Cref{sect:charac-consequences}}

\begin{delayedproof}{prop_hc-equiv-tri-partition}
    ~\\
    \textbf{A ``dimensional layer'' of a perfect \HC\ is a tri-partition:}\\
    \input{resources/proofs/prop_hc-induce-tri-partition}
    
    \bigskip
    \noindent\textbf{A ``stack'' of tri-partitions is a perfect \HC:}\\
    \input{resources/proofs/prop_tri-partition-induce-hc}
\end{delayedproof}

\subsection{Duality and Explicit Cohomology Bases}\label{sect:explicit-cohomology-bases}
    \begin{definition}[dual complex]\label{def:dual-complex}
        Given a complex $(K,\dr)$, its dual complex is the chain complex $(K^\ast, \dr^\ast)$ induced by the cells of $K$ but in reverse order, and equipped with the coboundary operator: $K^\ast_q := K_{n-q}$ and $\dr^\ast$ is the transpose of $\dr$.
        The $q$-homology of $(K,\dr)$ corresponds to the $(n-q)$-cohomology of $(K^\ast, \dr^\ast)$ and vice versa.
    \end{definition}

    \begin{proposition}\label{prop_hvdf-dual-reduction}
        If $(P,S,C)$ is a \HC\ for $(K,\dr)$ with reduction $(f,g,h)$, then $(S,P,C)$ is a \HC\ for the dual complex $(K^\ast, \dr^\ast)$ with reduction $(g^\ast,f^\ast,h^\ast)$.
        %More precisely, $K^\ast_q$ corresponds to the $(n-q)$-chains and $\dr^\ast$, $g^\ast$, $f^\ast$ and $h^\ast$ are the transpose of $\dr$, $g$, $f$ and $h$.
        %In particular, $(P,S,C)$ is perfect for $(K,\dr)$ if and only if $(S,P,C)$ is perfect for $(K^\ast, \dr^\ast)$.
    \end{proposition}

    The notation for the induced complex also transfers to $(K^\ast, \dr^\ast)$: given a set of cells $A$ in a complex $\cplx$, $(K^\ast(A),\dr^\ast)$ is the minimum sub-chain complex of $(K^\ast, \dr^\ast)$ containing $A$.
    Precisely, $K^\ast(A)_q$ is generated by the cofaces of dimension $q$ of cells in $A$.

    % Given a $q$-cohomology basis $(\f_i)_{i\leq\beta}$ for a complex $\cplx$ and $J \subseteq \{1, \dots, \beta\}$ a set of indices, we denote $K^{*J} := K^\ast\left(\bigcup_{i \in J} \ols{\f_i}\right)$.
    % We denote $\iota^J_q$ the $q$-homology map induced by the inclusion $(K^{*J},\dr) \subseteq (K^\ast,\dr)$.
    % With this notation, we also denote $(K^{*\beta},\dr) := (K^{*\{1, \dots, \beta\}},\dr)$ and $\iota_q^\beta$ its associated homology map.
    Given a $q$-cohomology basis $(\f_i)_{i\leq\beta}$ for a complex $\cplx$ (which is by duality a $(n-q)$-homology basis for $(K^\ast, \dr^\ast)$), we denote $(K^{*\beta},\dr^\ast) := \big(K^\ast\big(\bigcup\nolimits_{i \leq \beta} \ols{\f_i}\big),\dr^\ast\big)$ the sub-complex of $(K^\ast, \dr^\ast)$ induced by $(\f_i)_{i\leq\beta}$.

    \begin{definition}[\Explicit\ cohomology bases]\label{def:explicit-co-bases}
        A $q$-cohomology basis $(\f_i)_{i\leq\beta}$ for a complex $\cplx$ is said to be \emph{\explicit} if it satisfies the following property:
        $$
            \forall k,\ \ols{\f_k}\bs\bigcup\nolimits_{i \neq k} \ols{\f_i} \neq \emptyset \qquad \tx{and}\qquad
            dim \left( H_{n-q}\left(K^{*\beta}\right) \right) = \beta
        $$
        By definition, $(\f_i)_{i\leq\beta}$ is an \explicit\ $q$-cohomology basis for a complex $\cplx$ if and only if $(\f_i)_{i\leq\beta}$ is an \explicit\ $(n-q)$-homology basis for $(K^\ast, \dr^\ast)$.
    \end{definition}
    \Cref{prop_hvdf-dual-reduction,thm_explicit-generators-hdvf,thm_hdvf-generators-explicit} implies the following characterization of cohomology bases:
    \begin{theorem}
        Given a perfect \HC\ $(P,S,C)$ for $\cplx$ with reduction $(f,g,h)$, $(f_q^\ast(\gamma))_{\gamma\in C_q}$ is an \explicit\ cohomology basis.
        Conversely, given $(\f_i)_{i\leq\beta}$ an \explicit\ $q$-cohomology basis for a complex $\cplx$, there exists a perfect \HC\ $(P,S,C)$ for $\cplx$ with reduction $(f,g,h)$ such that $(f_q^\ast(\gamma))_{\gamma\in C_q} = (\f_i)_{i\leq\beta}$.
    \end{theorem}
    
% \begin{delayedproof}{prop_hc-induce-tri-partition}
%     \input{resources/proofs/prop_hc-induce-tri-partition}
% \end{delayedproof}

% \begin{delayedproof}{prop_tri-partition-induce-hc}
%     \input{resources/proofs/prop_tri-partition-induce-hc}
% \end{delayedproof}

% \subsection{Proofs of Persistent Homology and \HCs}\label{sect:appendix-hdvf-persistence}
% \input{resources/hdvf-persistence}

\end{document}

%% file: resources/tikz/reduction-diagram.tex
% https://tikzcd.yichuanshen.de/#N4Igdg9gJgpgziAXAbVABwnAlgFyxMJZARgBoAGAXVJADcBDAGwFcYkQBpAfWAEcBfEP1LpMufIRQAmCtTpNW7bnwC0xQcNHY8BIuVk0GLNohAAdM1Ag4EmkBm0SiAZgPzj7C1ZtCR9sTqSJKRSckaKphwA5DwCvlriutIhYQomnDGq6vH+jknI+qGGaZ6W1rZ+DolBrkXuEeZlPvxyMFAA5vBEoABmAE4QALZIMiA4EEj69ekWaPR9eEyxANTZNHAAFlg9OEhqdv1DkzTjSGTTpXMLWEt8gutbO3vqfofDiOeniK4XprPzi0YsX2IAe212iH2rwG7ymXwALGCnpDzuF0j0uLxQSBGPQAEYwRgABQCTlMjBgTwOMLOJwmiAArEiIWpih5TBishpoUdEIixvSpptwc82Q12picm8kEyBbSQMLkazfiAJVypTTEAA2OlIflo9hQFbZam8-lfWUG0xGrGm96yr4AdjF6SN6rtSB+X3OipZqJKpg2wKk3N6ms+guZopVQfdPNhusQo190atICDtso-CAA
\begin{tikzcd}
\dots \arrow[rr, "\partial_{q+1}", shift left] && K_{q} \arrow[rr, "\partial_{q}", shift left] \arrow[dd, "f_q", shift left] \arrow[ll, "h_q", shift left] && K_{q-1} \arrow[rr, "\partial_{q-1}", shift left] \arrow[dd, "f_{q-1}", shift left] \arrow[ll, "h_{q-1}", shift left] && \dots \arrow[ll, "h_{q-2}", shift left] \\
                                              &                                                                                                        &                                                                                                                    &                                        \\
\dots \arrow[rr, "d_{q+1}"]                    && K'_{q} \arrow[uu, "g_q", shift left] \arrow[rr, "d_q"]                                                  && K'_{q-1} \arrow[uu, "g_{q-1}", shift left] \arrow[rr, "d_{q-1}"]                                                    && \dots
\end{tikzcd}

%% file: resources/proofs/prop_elementary-generators.tex
% !TEX root = ../../main.tex
Let $(\g_i)_{i\leq \beta}$ be an \explicit\ $q$-homology basis of $(K,\dr)$ and $a$, $b$ be cycles such that $\ols{a}\sqcup \ols{b} = \ols{\g_k}$ (and so $a+b = \g_k$).
Denote $J := \{k\}$.
We get that $a$ and $b$ are in $\ker\dr \cap K^J_q$ so by~\Cref{prop_explicit-bases-equivalence}(2.) they are in $\Span(\g_k) = \{0, \g_k\}$.
Hence, either $a$ or $b$ is $0$, otherwise we would have $\g_k = \g_k + \g_k = 0$.

%% file: resources/reduction-example.tex
Let us state~\Cref{prop_hvdf-reduction} again: a \HC\ $(\green{P},\red{S},\blue{C})$ induces a reduction $(f,g,h)$ from $(K,\dr)$ to $(Span \blue{C},d)$ with:
    \begin{align*}
      g :=
      \begin{blockarray}{cc}
      \blue{C} \\
      \begin{block}{(c)c}
      0 & \ \green{P} \\
      G & \ \red{S} \\
      id & \ \blue{C} \\
      \end{block}
      \end{blockarray}
      &&f :=
      \begin{blockarray}{cccc}
      \green{P} & \red{S} & \blue{C} \\
      \begin{block}{(ccc)c}
      F & 0 & id & \ \blue{C} \\
      \end{block}
      \end{blockarray}
      &&h :=
      \begin{blockarray}{cccc}
      \green{P} & \red{S} & \blue{C} \\
      \begin{block}{(ccc)c}
      0 & 0 & 0 & \ \green{P} \\
      H & 0 & 0 & \ \red{S} \\
      0 & 0 & 0 & \ \blue{C} \\
      \end{block}
      \end{blockarray} %\\
      &&d :=
      \begin{blockarray}{cc}
      \blue{C} \\
      \begin{block}{(c)c}
      D & \ \blue{C} \\
      \end{block}
      \end{blockarray}
    \end{align*}
where\ $H=(\dr_{\red{S}\green{P}})\mun$,\ \ $F = -\dr_{\red{S}\blue{C}} \cdot H$,\ \ $G = -H \cdot \dr_{\blue{C}\green{P}}$\ and\ $D = \dr_{\blue{C}\blue{C}} - \dr_{\red{S}\blue{C}} \cdot H \cdot \dr_{\blue{C}\green{P}}$

\begin{figure}[htb!]
    \centering
    \includegraphics[width = 0.33\textwidth]{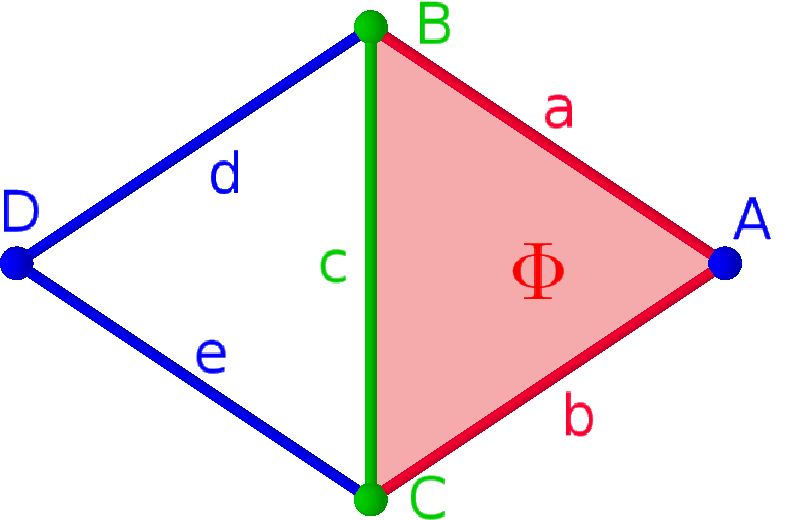}
    \hspace{0.2\textwidth}
    \includegraphics[width = 0.22\textwidth]{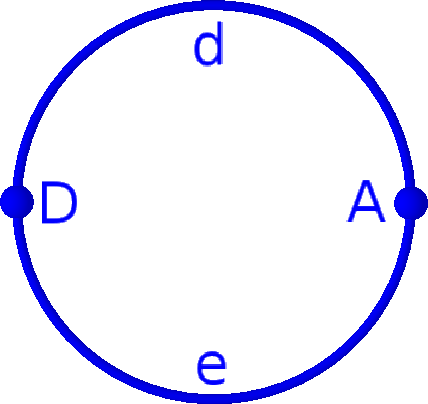}
    \caption{An example of a \HC\ $(\green{P},\red{S},\blue{C})$ on a complex $(K,\dr)$ (on the left) and a representation of the reduced complex $(\Span\blue{C}, d)$ induced by the reduction of this \HC (on the right).}
    \label{fig:reduction-example}
\end{figure}
In the example of~\Cref{fig:reduction-example},~\Cref{prop_hvdf-reduction} implies the following matrices:
\begin{align*}
  \dr_1 =
  \begin{blockarray}{cccccc}
    \secondary{a} & \secondary{b} & \primary{c}  & \critical{d}  & \critical{e} \\
    \begin{block}{(ccccc)c}
    1 & 1 & 0 & 0 & 0 & \ \critical{A} \\
    1 & 0 & 1 & 1 & 0 & \ \primary{B} \\
    0 & 1 & 1 & 0 & 1 & \ \primary{C} \\
    0 & 0 & 0 & 1 & 1 & \ \critical{D} \\
    \end{block}
  \end{blockarray}
  &&\dr_2 =
  \begin{blockarray}{cc}
    \red{\Phi} \\
    \begin{block}{(c)c}
    1 & \ \secondary{a} \\
    1 & \ \secondary{b} \\
    1 & \ \primary{c} \\
    0 & \ \critical{d} \\
    0 & \ \critical{e} \\
    \end{block}
  \end{blockarray}
  && H_1 =
  \begin{blockarray}{cc}
    \primary{c} \\
    \begin{block}{(c)c}
    1 & \ \red{\Phi} \\
    \end{block}
  \end{blockarray}
  &&H_0 =
  \begin{blockarray}{ccc}
    \primary{B}  & \primary{C} \\
    \begin{block}{(cc)c}
    1 & 0 & \ \secondary{a} \\
    0 & 1 & \ \secondary{b} \\
    \end{block}
  \end{blockarray}
\end{align*}
\vspace{-0.1\textwidth}
\begin{align*}
  d_1 &= D_1 = \dr_{\critical{C_1}\critical{C_0}} - \dr_{\secondary{S_1}\critical{C_0}} \cdot H_0 \cdot \dr_{\critical{C_1}\primary{P_0}} 
  && d_2 = 0 
  && d_0 = \begin{blockarray}{cc}
    \critical{A} & \critical{D}\\
    \begin{block}{(cc)}
     0 & 0 \\
    \end{block}
  \end{blockarray}\\
  &=\begin{pmatrix}
    0 & 0  \\
    1 & 1 
  \end{pmatrix}
    -
    \begin{pmatrix}
    1 & 1  \\
    0 & 0 
  \end{pmatrix}
    \cdot
    \begin{pmatrix}
    1 & 0 \\
    0 & 1
  \end{pmatrix}
    \cdot
    \begin{pmatrix}
    1 & 0 \\
    0 & 1
    \end{pmatrix} \\
  &=\ 
  \begin{blockarray}{ccc}
    \critical{d} & \critical{e} \\
    \begin{block}{(cc)c}
    1 & 1 & \ \critical{A} \\
    1 & 1 & \ \critical{D} \\
    \end{block}
  \end{blockarray}
\end{align*}

As reductions preserve homology, the homology of $(K,\dr)$ is the same as the homology of $(\Span\blue{C}, d)$.
In particular:
$$H_1(K) \approx \frac{\ker{d_1}}{\im{d_2}} = \frac{\Span{\{\critical{d}+\critical{e}\}}}{\{0\}} = \Span{\{\critical{d}+\critical{e}\}} \approx \Z/2\Z$$

%% file: resources/proofs/lem_canonical-cycle.tex
\add{We only prove the first equation as the second equation is the same result stated on the dual complex (see~\Cref{def:dual-complex,prop_hvdf-dual-reduction} for more details on this duality).}

First, $z_q(x) \in x+\Span S$ because $z_q(x) = x - h\dr(x)$ and $\im h \subseteq \Span S$.
Moreover, $z_q(x) \in \ker\dr$ because $z_q(x) = gf(x)+\dr h(x)$ which is a sum of cycles because the \HC\ is perfect.

Second, suppose $x+s\in (x+\Span S)\cap\ker\dr$.
This implies that $y := z_q(x) - (x+s)$ is a cycle in $\Span S$.
With the reduction properties we have
\add{$y = gf(y) + \dr h(y) + h \dr(y) = 0$}
(because $h\circ\inc{S}{} = 0$,\ $\dr(y)=0$ and $f\circ\inc{ S}{} = 0$, see~\Cref{def:reduction}).
Hence, $x+s = z_q(x)$.

%% file: resources/proofs/prop_explicit-bases-equivalence.tex
% !TEX root = ../../main.tex
For simplicity we write $[\cdot]$ instead of $[\cdot]^K_q$ and $[\cdot]^J$ instead of $[\cdot]^{K^J}_q$.
Firstly, we state some results about the $q$-homology basis $(\g_i)_{i\leq\beta}$:
\begin{itemize}
    \item $([\g_i])_{i\leq\beta}$ is a basis of $H_q(K)$ (by~\Cref{def:generators-homology-bases}).
    \item $\iota_q^J([x]^J) = [x]$ for every $x$ in $K^J$ (by functoriality of the homology functor).
    % \item For every $k$, $(\g_i)_{i\in J}$ is a free family.
    %     Otherwise we would have $\g_j \in \Span((\g_i)_{i\neq j})$ which would imply that $([\g_i])_{i\leq\beta}$ is not a basis of $H_q(K)$.
    \item For every $J$, $([\g_i]^J)_{i\leq J}$ is linearly independent.
        Indeed, $\sum \epsilon_i [\g_i]^J = 0$ implies $\sum \epsilon_i [\g_i] = 0$, so every $\epsilon_i$ is 0 because $([\g_i])_{i\leq\beta}$ is a basis.
\end{itemize}
\medskip

\noindent\textbf{1. $\implies$ 2.}:\quad
    Let $x \in \ker\dr \cap K^J_q$, let us show that it is in $\Span \left( (\g_i)_{i\in J} \right)$
    
    \smallskip

    $\iota_q^J$ has rank $|J|$ and $([\g_i])_{i\in J} = (\iota_q^J([\g_i]^J))_{i\leq J}$ is a linearly independent set of cardinal $|J|$ in $\im\iota_q^J$, therefore $([\g_i])_{i\leq J}$ is a basis of $\im\iota_q^J$.
    Hence, \add{there exist $(\epsilon_i)_{i\in J}$ coefficients} such that:
    $$\iota_q([x]^J) = \sum\nolimits_{i\in J} \epsilon_i [\g_i] = \iota_q \left( \left[\sum\nolimits_{i\in J} \epsilon_i \g_i\right]^J \right) $$
    By injectivity of $\iota_q$, we have $[x]^J = \left[\sum_{i\in J} \epsilon_i \g_i\right]^J$, therefore there exists a $(q+1)$-chain $y\in K^J$ such that $x = \sum_{i\in J} \epsilon_i \g_i + \dr y$.
    However, the maximal dimension of \add{cells in} $K^J$ is $q$, so $y=0$.
    As a result, $x = \sum_{i\in J} \epsilon_i \g_i \in \Span \left( (\g_i)_{i\in J} \right)$.

\medskip
\noindent\textbf{2. $\implies$ 3.}:\quad
    Firstly, by contradiction suppose that $\ols{\g_k}\bs\bigcup\nolimits_{i \neq k} \ols{\g_i} = \emptyset$.
    
    \smallskip

    Let $J := \{1,\dots, \beta\}\bs\{k\}$.
    We then have $\ols{\g_k} \subseteq \bigcup\nolimits_{i \in J} \ols{\g_i} = \ols{K^J}$, therefore $\g_k \in K^J$.
    In addition, $\g_k$ is a $q$-cycle (because it is a $q$-generator).
    By assumption, we obtain that $\g_k \in \Span\left( (\g_i)_{i\in J} \right)$ which contradicts the fact that $([\g_i]^\beta)_{i\leq\beta}$ is linearly independent.
    As a consequence $\ols{\g_k}\bs\bigcup\nolimits_{i \neq k} \ols{\g_i} \neq \emptyset$.\\
    
    Then, $([\g_i]^\beta)_{i\leq \beta}$ is linearly independent, let us show it is a basis for $H_q(K^\beta)$.\\
    Let $[x]^\beta \in H_q(K^\beta)$.
    We have that $x$ is a $q$-cycle in $K^\beta$.
    By assumption we get that $x \in \Span((\g_i)_{i\leq \beta})$, therefore $[x]^\beta \in \Span(([\g_i]^\beta)_{i\leq \beta})$.
    Hence, $([\g_i]^\beta)_{i\leq \beta}$ generates $H_q(K^\beta)$ and is linearly independent so it is a basis for $H_q(K^\beta)$.
    As a result $dim(H_q(K^\beta)) = \beta$.

\medskip
\noindent\textbf{3. $\implies$ 1.}:\quad
    Given $J$, $([\g_i]^J)_{i\in J}$ is linearly independent, let us show that it is a basis for $H_q(K^J)$.
    
    \smallskip

    \add{Consider a $q$-cycle $x$ in $K^J$ and its associated homology class $[x]^J \in H_q(K^J)$}.
    Let $\iota_q^{J \to \beta} : H_q(K^J) \to H_q(K^\beta)$ be the homology morphism induced by the inclusion $(K^J,\dr)\subseteq (K^\beta,\dr)$.
    We have $\iota_q^{J \to \beta}([x]^J)=[x]^\beta$.
    
    \smallskip

    $H_q(K^\beta)$ has dimension $\beta$ by assumption. $([\g_i]^\beta)_{i\leq\beta}$ is a linearly independent set of $H_q(K^\beta)$ with cardinal $\beta$ so it is a basis of it.

    Hence, \add{there exist $(\epsilon_i)_{i\in J}$ coefficients such that} $[x]^\beta = \sum_{i\leq \beta} \epsilon_i [\g_i]^\beta$.
    Thus, there exists a $(q+1)$-chain $y\in K^\beta$ such that $x = \sum_{i\leq \beta} \epsilon_i \g_i + \dr y$.
    As $K^\beta_{q+1} = \{0\}$ we obtain $y=0$ and $x = \sum_{i\leq \beta} \epsilon_i \g_i$.

    \smallskip

    By contradiction: suppose that there is $j \notin J$ such that $\epsilon_j \neq 0$.
    Then, as $\ols{\g_j}\bs\bigcup_{i \neq j} \ols{\g_i} \neq \emptyset$, let $\gamma$ be a cell in $\ols{\g_j}$ that is not in the other $\ols{\g_i}$.
    We have:
    \begin{align*}
        \langle x, \gamma \rangle &= \sum\nolimits_{i\leq\beta} \epsilon_i \langle \g_i, \gamma \rangle
        = \epsilon_j \langle \g_j, \gamma \rangle \neq 0 \qquad \tx{because $\gamma \in \ols{\g_j}$ and $\epsilon_j \neq 0$}
    \end{align*}
    Hence, $\gamma \in \ols{x}$.
    However, $x \in K^J$ so $\ols{x} \subseteq \ols{K^J} = \bigcup_{i\in J} \ols{\g_i}$ but $\gamma$ is not in any $\ols{\g_i}$ with $i\in J$ by definition.
    As a consequence, $\forall j \notin J, \epsilon_j = 0$.
    
    \smallskip

    This implies that $x = \sum_{i\in J} \epsilon_i \g_i$ and $[x]^J \in \Span(([\g_i]^J)_{i\in J})$.\\
    As a result $([\g_i]^J)_{i\in J}$ is a basis for $H_q(K^J)$ with cardinal $|J|$.
    Therefore, as $([\g_i])_{i\in J} = (\iota_q^J([\g_i]^J))_{i\in J}$ is a linearly independent family of $\im \iota_q^J$ we obtain that $\iota_q^J$ is an injective map with rank $|J|$.

%% file: resources/proofs/thm_hdvf-generators-explicit.tex
% !TEX root = ../../main.tex
Denote $C_q = \{\gamma_1, \dots, \gamma_\beta\}$.
For this proof, we use the second characterization of an \explicit\ homology basis in~\Cref{prop_explicit-bases-equivalence}.
Let $J\subseteq \{1,\dots, \beta\}$ and $K^J := K\left(\bigcup_{i \in J} \ols{g(\gamma_i)}\right)$. Let $x\in\ker\dr \cap K^J_q$, we will show that $x\in\Span{(g(\gamma_i))_{i\in J}}$.

As $\im g \subseteq \Span P^\bot$ (see~\Cref{prop_hvdf-reduction}), we have that $\ols{x} \subseteq \ols{K^J_q} = \bigcup\nolimits_{i \in J} \ols{g(\gamma_i)} \subseteq P^\bot$.
Hence $x\in\Span P^\bot$ and there exists $s\in\Span S$ and $c\in\Span C$ such that $x = c + s$ and  $\ols{x} = \ols{c}\sqcup\ols{s}$.

By~\Cref{lem_canonical-cycle} we have that for all $i$, $(g(\gamma_i)- \gamma_i) \in \Span S$ so $\ols{g(\gamma_i)}\cap C = \{\gamma_i\}$.
This implies:
$$\ols{K^J_q} \cap C = \bigcup\nolimits_{i \in J} \ols{g(\gamma_i)} \cap C = \bigcup\nolimits_{i \in J} \{\gamma_i\}$$

As $x\in K^J_q$, we obtain $\ols{x} \cap C = \ols{c} \subseteq \{\gamma_i,\ i\in J\}$.
This implies that there exist $(\epsilon_i)_{i\in J}$ elements of the field $\mathcal{F}$ such that $c = \sum_{i \in J} \epsilon_i \gamma_i$.
Thus with reduction properties (see~\Cref{def:reduction}):
\begin{align*}
    x &= gf(x) + \dr h (x) + h \dr(x)\\
     &= gf(c+s) &&\tx{as $h \circ \inc{ P^\bot}{} = 0$ and $x\in\Span P^\bot$ and $x\in\ker\dr$}\\
     %&= g(f(c)+f(s)) &&\tx{because $x\in\ker\dr$}\\
     &= g(c) &&\tx{as $f \circ \inc{ C}{} = \id_{ C}$ and $f\circ \inc{ S}{} = 0$}\\
    &= \sum\nolimits_{i\in J} \epsilon_i g(\gamma_i) &&\tx{As a result $x\in\Span{(g(\gamma_i))_{i\in J}}$.}
\end{align*}

%% file: resources/proofs/lem_injective-hdvf-completion.tex
% !TEX root = ../../main.tex
Using the completion~\Cref{coro:hdvf-completion}, there exists $X''=(P'',S'',C'')$ a perfect \HC\ for $\cplxpdr$ with reduction $(f'',g'',h'')$ such that $P\subseteq P''$ and $S\subseteq S''$, but a priori $C^*$ is not included in $C''$.

The idea of the proof is to first show the injectivity of $\proj{}{C''\bs C^*}\circ f''\circ \inc{P''\cap C^*}{}$ and then use it to find a set $\Gamma \subseteq C''\bs C^\ast$ such that $X' := \M_{P'' \cap C^\ast}^\Gamma X''$ is a valid operation.

\begin{enumerate}
% 1. f'' i_C* injective
\item First, we show that $f''\circ \inc{C^*}{}$ is injective:\\
    Suppose that there is $c\in \Span C^*$ such that $f''(c) = 0$.
    By~\Cref{lem_canonical-cycle}, we write $g(c) = c+s$ with $s\in\Span S \subseteq \Span S''$.
    With~\Cref{def:reduction}-5. we have:
    \begin{align*}
        g(c) &= g'' f'' g(c) + \dr h'' g(c) + h''\dr g(c)\\
        &= g'' f'' (c+s) + \dr h'' g(c) && \tx{because $\dr g = 0$}\\
        &= g''f''(c) + g''f''(s) + \dr h'' g(c)\\
        &= g''f''(c) + \dr h'' g(c) && \tx{as $f''\circ\inc{S''}{} = 0$ and $s\in\Span S''$}\\
        &= \dr h'' g(c) \in \dr(K') &&\tx{because $f''(c) = 0$}
    \end{align*}
    Hence, $g(c)$ is a boundary in $K'$ so $[g(c)]^{K'} = 0$.
    By assumption, $( [g_q(\gamma)]^{K'} )_{\gamma \in C^\ast}$ is linearly independent, and $c\in\Span C^\ast$, therefore we obtain that $c = 0$.

\smallskip    
% 2. j_C\C'' f'' i_C''C injective
\item Second, we show the injectivity of $\proj{}{C''\bs C^*}\circ f''\circ \inc{P''\cap C^*}{}$:\\
    Suppose that there is $c\in \Span (P''\cap C^*)$ such that $\proj{}{C''\bs C^*} f''(c) = 0$.
    As $\im f'' \subseteq \Span C''$, this implies that $f''(c)\in\Span(C''\cap C^*)$.
    
    As $f'' \circ \inc{C''}{} = \id_{C''}$ (see~\Cref{prop_hvdf-reduction}), we obtain that $f''(f''(c)) = f''(c)$.
    As a consequence, we have $f''(c - f''(c)) = 0$ which implies $c = f''(c)$ by injectivity of $f''\circ \inc{C^*}{}$ (1.).
    However, we have $c \in \Span (P''\cap C^*)$ and $f''(c)\in\Span(C''\cap C^*)$, so $c = 0$ (because $P''\cap C'' = \emptyset$).

\smallskip
% 3. C* = C*P U C*C''
\item Third, we show that $C^\ast = (P'' \cap C^\ast) \sqcup (C'' \cap C^\ast)$:\\
    By contradiction suppose $\gamma \in S''\cap C^\ast$.
    \Cref{lem_canonical-cycle} implies that $g(\gamma) \in \gamma+\Span S \subseteq \gamma + \Span S''$.
    Moreover, $g(\gamma)$ is a cycle (because $X$ is perfect) so we have $g(\gamma) \in \ker\dr\cap(\gamma +\Span S'')$.
    Therefore, $g(\gamma) = z''_q(\gamma)$ by~\Cref{lem_canonical-cycle}.
    In addition, we have $z''_q \circ \inc{S''}{} = 0$ (by~\Cref{lem_canonical-cycle} again) and $\gamma\in S''$ so $g(\gamma) = 0$, which is not possible by injectivity of $g$.
    As a result, $S''\cap C^\ast = \emptyset$ and we obtain $C^\ast = (C^\ast \cap P'') \sqcup (C^\ast \cap C'')$ because $(P'',S'',C'')$ is a partition.
\end{enumerate}
% \smallskip
% 4. conclusion defining gamma and M
Finally, we can define the suited \HC\ $X'$:\\
    As $\proj{}{C''\bs C^*}\circ f''\circ \inc{P''\cap C^*}{}$ is injective (2.), it admits a invertible sub-matrix of maximal rank.
    Hence, there exists a subset $\Gamma \subseteq C''\bs C^\ast$ such that $\proj{C''\bs C^\ast}{\Gamma} \circ \proj{}{C''\bs C^*}f''\inc{P''\cap C^*}{}$ is an invertible square matrix.
    Therefore, $\proj{C''\bs C^\ast}{\Gamma} f'' \inc{P''\cap C^*}{C^\ast}$ is invertible, so by~\Cref{def:generalized-hdvf-operators} $\M_{P''\cap C^*}^\Gamma$ is a valid operation for $X''$.
    
    Let $X' = (P',S',C') := \M_{P''\cap C^*}^\Gamma X''$, we have:
    $$ P' = (P''\sqcup \Gamma) \bs ({P''\cap C^*}) \qquad S' = S'' \qquad C' = \left(C''\sqcup ({P''\cap C^*})\right) \bs \Gamma  $$
    \begin{itemize}
        \item $P \subseteq P'$ because $P\subseteq P''$ and $P \cap ({P''\cap C^*}) = \emptyset$.
        \item $S \subseteq S'$ because $S \subseteq S''$.
        \item $C^* \subseteq C'$ because $ C^* = ({C''\cap C^*})\sqcup ({P''\cap C^*})$ (3.) and $\Gamma \cap C^\ast = \emptyset$.% (as $\Gamma \subseteq C''\bs C^*$)
    \end{itemize}

%% file: resources/proofs/lem_explicit-generator-complex.tex
% !TEX root = ../../main.tex
For simplicity we will write $K^k$ instead of $K^{\{1\dots k\}}$ and $[\cdot]^k$ instead of $[\cdot]^{K^k}_q$.
We will also write $\iota^k$ instead of $\iota^{K^k}_q$ and omit the $q$ subscripts for reduction maps.

\smallskip

Let $(\g_i)_{i\leq \beta}$ an explicit $q$-homology basis for $\cplx$.
% First let us define the $q$-homology skeleton complex $(K^k,\dr)$:
% $$K^k := K \left( \bigcup_{i\leq k} \ols{\g_i} \right)$$ % already define in the paper
Let $k\leq \beta$.
Because $\ols{\g_k}\bs\bigcup_{i \neq k} \ols{\g_i} \neq \emptyset$, let us fix an element $\gamma_k$ in $\ols{\g_k}$ that is not in the other $\ols{\g_i}$s.

\smallskip

We will build by induction for $k = 0, \dots, \beta$ a perfect \HC\ $X^k = (P^k,S^k,C^k)$ for the complex $(K^k,\dr)$ (with reduction $(f^k,g^k,h^k)$) such that $C^k_q = \{ \gamma_i,\ i \leq k\}$ and $\forall i \leq k,\ g^k(\gamma_i) = \g_i$.

\medskip
\noindent\textbf{Base case}: $k = 0$.\\
    $K^0 = \{0\}$, $X^0 = (\emptyset,\emptyset,\emptyset)$ with reduction $(f^0,g^0,h^0)$, such that $f^0$, $g^0$ and $h^0$ are the zero function $\{0\}\to\{0\}$.

\medskip
\noindent\textbf{Induction step}: $k+1$.\\
Suppose that there exists $X^k = (P^k,S^k,C^k)$ a perfect \HC\ for the complex $(K^k,\dr)$ (with reduction $(f^k,g^k,h^k)$) such that $C^k_q = \{ \gamma_i,\ i \leq k\}$ and $\forall i \leq k,\ g^k(\gamma_i) = \g_i$.

We have $\ols{K^{k+1}} = \ols{K^k} \cup \ols{\g_{k+1}}$.
Denote $\iota^k$, $\iota^{k+1}$ and $\iota^{k\to k+1}$ the following $q$-homology maps induced the inclusions $(K^k,\dr) \subseteq (K,\dr) $,\quad $(K^{k+1},\dr) \subseteq (K,\dr) $ and $(K^k,\dr) \subseteq (K^{k+1},\dr)$:
$$\iota^k:H_q(K^k)\to H_q(K) \quad \iota^{k+1}:H_q(K^{k+1})\to H_q(K) \quad \iota^{k\to k+1}:H_q(K^k) \to H_q(K^{k+1})$$
By functoriality of the homology functor we have $ \iota^k = \iota^{k+1} \circ \iota^{k\to k+1} $.

\smallskip

$(\g_i)_{i\leq \beta}$ is \explicit\ so $\iota^k$ is injective by~\Cref{prop_explicit-bases-equivalence}(1.), therefore $\iota^{k\to k+1}$ is also injective.
In addition, $([g^k(\gamma)]^{k})_{\gamma \in C^{k}_q}$ is linearly independent (because $X^k$ is perfect).
Hence, $([g^k(\gamma)]^{k+1})_{\gamma \in C^{k}_q} = (\iota^{k\to k+1}[g^k(\gamma)]^{k})_{\gamma \in C^{k}_q}$ is linearly independent.

Using~\Cref{lem_injective-hdvf-completion} with $C^\ast := C^k_q$, it is then possible to complete $X^k$ into a perfect \HC\ $X = (P,S,C)$
for $(K^{k+1}, \dr)$ such that $P^k\subseteq P$ and $S^k \subseteq S$ and $C^k_q \subseteq C_q$. We will now slightly modify $X$ to have the desired properties.

    \smallskip

$(\g_i)_{i\leq \beta}$ is \explicit\ so $\iota^{k+1}$ is injective with rank $k+1$ by~\Cref{prop_explicit-bases-equivalence}, therefore $|C_q| = k+1$.
As $|C^k_q| = k$ there exists a cell $\tau \in C_q$ such that $C_q = C^k_q \sqcup \{\tau\}$.
As $\tau$ is not in $\ols{K^k}$, we have that $\tau \in \ols{\g_{k+1}}$.

\add{The maximal dimension of cells in the complex} $(K^{k+1},\dr)$ is $q$ by construction, therefore every $q$-cell of it is either a secondary or a critical cell (otherwise $\dr_{SP}$ would not be invertible).
Hence, we have that $\ols{\g_{k+1}} \bs \{\tau\} \subseteq S$ because $\ols{\g_{k+1}}$ cannot contain critical cells from $C^k_q = \{ \gamma_i,\ i \leq k\}$.
This implies that $\g_{k+1} \in \tau + \Span S$ so by~\Cref{lem_canonical-cycle} we have $g(\tau) = \g_{k+1}$.

    \smallskip

If $\tau = \gamma_{k+1}$, we can just define $X^{k+1} := X$. If that is not the case we use a $\W$ operation:
as $\gamma_{k+1} \in \ols{\g_{k+1}} = \ols{g(\tau)}$, $\W_{\{\gamma_{k+1}\}}^{\{\tau\}}$ is a valid operation for $X$ \add{(see~\Cref{def:generalized-hdvf-operators})}.
As a result we define $X^{k+1} := \W_{\{\gamma_{k+1}\}}^{\{\tau\}} X = (P^{k+1},S^{k+1},C^{k+1})$, which is a perfect \HC\ for $K^{k+1}$ \add{because a valid $\W$ operation transforms a perfect \HC\ into another perfect \HC (see~\Cref{def:generalized-hdvf-operators}).}
We denote $(f^{k+1},g^{k+1},h^{k+1})$ its reduction.

    \smallskip
    
We now verify that $X^{k+1}$ satisfies the desired properties.\\
Firstly, $C^{k+1}_q = \{ \gamma_i,\ i \leq k+1\}$.
Secondly, as $P^{k} \subseteq P^{k+1}$, $S^{k} \subseteq S^{k+1}$ and $C^k_q \subseteq C \cap C^{k+1}$,~\Cref{lem_generator-preservation} implies that $\forall i \leq k,\ g^{k+1}(\gamma_i) = g^k(\gamma_i) = \g_i$.
%In addition, the $\W_{\gamma_{k+1}}^\tau$ operation preserves the homology generator associated to $\tau$ (see~\Cref{prop_operator-generator-preservation}) so we also have $g^{k+1}(\gamma_{k+1}) = \g_{k+1}$.
In addition we also have $g^{k+1}(\gamma_{k+1}) = \g_{k+1}$.
Indeed, as $\tau$ and $\gamma_{k+1}$ are both cells of $\g_{k+1} = g(\tau)\in \tau + \Span S$, there exists a $s\in\Span S$ such that $\g_{k+1} = \gamma_{k+1} + \tau + s$.
We have $S^{k+1} = S \sqcup \{\tau\}\bs\{\gamma_{k+1}\}$, therefore $\g_{k+1} \in \gamma_{k+1} + \Span S^{k+1}$.
By~\Cref{lem_canonical-cycle} we get $g^{k+1}(\gamma_{k+1}) = \g_{k+1}$.

%% file: resources/proofs/thm_explicit-generators-hdvf.tex
% !TEX root = ../../main.tex
% Idea:
% \begin{itemize}
%     \item build the perfect \HC\ $X^\beta=(P^\beta,S^\beta,C^\beta)$ for $K^{\beta} := K\left(\bigcup \ols{\g_i}\right)$ ;
%     \item prove that $\iota^\beta_q \circ \left[g^\beta\right]_q^{K^{\beta}} : C_q^\beta \to H_q(K)$ is injective because of the \explicit\ property ;
%     \item use~\Cref{lem_injective-hdvf-completion} to build a perfect \HC\ $X=(P,S,C)$ such that $P^\beta \subseteq P$, $S^\beta \subseteq S$ and $C_q^\beta = C_q$ ;
%     \item use~\Cref{lem_generator-preservation} with $X^\beta$ and $X$ to prove that $(g(\gamma))_{\gamma\in C_q} = (g^\beta(\gamma))_{\gamma\in C^\beta_q} = (\g_i)_{i\leq \beta}$.
% \end{itemize}
Using~\Cref{lem_explicit-generator-complex}, there exists $X^\beta=(P^\beta,S^\beta,C^\beta)$ a perfect \HC\ for $(K^{\beta}, \dr)$ (with reduction $(f^\beta,g^\beta,h^\beta)$) such that its associated $q$-homology basis is $(\g_i)_{i\leq \beta}$.

As $(\g_i)_{i\leq \beta}$ is \explicit, we have that $\iota_q^\beta : H_q(K^\beta)\to H_q(K)$ is injective (see~\Cref{prop_explicit-bases-equivalence}(1.)).
$\big([g^\beta(\gamma)]^\beta\big)_{\gamma\in C_q^\beta}$ is linearly independent (because $X^\beta$ is perfect), so we obtain that $\big(\big[g^\beta_q(\gamma)\big]^K\big)_{\gamma\in C_q^\beta} = \big(\iota_q^\beta\big[g^\beta_q(\gamma)\big]^\beta\big)_{\gamma\in C_q^\beta}$ is linearly independent. %that $\iota^\beta_q \circ \left[g^\beta\right]_q^{\beta} \circ \inc{C_q^\beta}{}$ is injective.

As a consequence we can use~\Cref{lem_injective-hdvf-completion} with $C^\ast := C_q^\beta$ to complete $X^\beta$ into a perfect \HC\ $X=(P,S,C)$ for $(K,\dr)$, with reduction $(f,g,h)$, such that $P^\beta \subseteq P$, $S^\beta \subseteq S$ and $C_q^\beta \subseteq C_q$.
As $\beta = dim(H_q(K)) = |C_q|$ and $\beta = dim(H_q(K^\beta)) = |C_q^\beta|$ we get $C_q^\beta = C_q$.
Using~\Cref{lem_generator-preservation} with $X^\beta$ and $X$ we obtain
$(g_q(\gamma))_{\gamma\in C_q} = (g^\beta_q(\gamma))_{\gamma\in C^\beta_q} = (\g_i)_{i\leq \beta}$.

%% file: resources/proofs/prop_hc-induce-tri-partition.tex
By~\Cref{lem_canonical-cycle}, we have $\Span S \cap \ker\dr = \{0\}$, so $S_q$ is a $q$-tree.
Similarly, $P_q$ is a $q$-cotree, moreover $|C_q|=\beta_q$ because $(P,S,C)$ is perfect.
    
    It remains to prove the maximality of $S_q$ and $P_q$.
    Let us consider $S_q$ first: we prove that if we add a cell $\tau$ to $S_q$, it creates a cycle.
    Let $\tau$ be a $q$-cell that is not in $S_q$.
    As $(P_q,S_q,C_q)$ is a partition of the $q$-cells, $\tau \in P_q\sqcup C_q$.
    
    Consider $x := gf(\tau)+\dr h(\tau) = \tau - h\dr(\tau)$ (by~\Cref{def:reduction}).
    $(P,S,C)$ is perfect so $\dr g = gd = 0$.
    This implies that $gf(\tau)$ is a cycle.
    It follows that $x$ is also a cycle because $gf(\tau)$ and $\dr h(\tau)$ are cycles.

    In addition, $h\dr(\tau) \in \Span S_q$ (by~\Cref{prop_hvdf-reduction}) but $\tau \notin S_q$ so $x= \tau - h\dr(\tau)$ is not zero and is composed of cells in $S_q\sqcup\{\tau\}$.
    Thus, $x$ is a non zero cycle of $S_q\sqcup\{\tau\}$, so $S_q$ is maximal.
    \smallskip
    
    The proof of the maximality of $P_q$ is similar:
    if $\tau \notin P_q$, then $f^\ast g^\ast(\tau)+\dr^\ast h^\ast(\tau)$ is a non zero cocycle with cells in $P_q\sqcup\{\tau\}$.

%% file: resources/proofs/prop_tri-partition-induce-hc.tex
We first prove that $(P,S,C) = (\bigcup A^q, \bigcup A_q, \bigcup E_q)$ is a \HC.
    To do that we need to show that $\dr_{|A^q}^{|A_{q+1}}$ is a squared invertible matrix for every $q$.

    To begin with we prove the injectivity of $\dr_{|A^q}^{|A_{q+1}}$:
    let $x \in \ker \dr_{|A^q}^{|A_{q+1}}$.
    We have that $x \in \Span A_{q+1}$ and that $\dr(x) \in \Span A_q \oplus \Span E_q$.
    Hence there exist $t \in \Span A_q$ (a tree) and $e\in \Span E_q$ such that $\dr(x) = t+e$.

    By~\Cref{def:tri-partition-canonical-cycle}, we have that $z_q(e)-e \in \Span A_q$ and $z_q(e)\in\ker\dr$.
    This implies:
    \begin{align*}
        %\dr(x) &= t+e\\
        \dr(x)-e &= t\\
        \dr(x)-e -(z_q(e)-e) &= t - (z_q(e)-e)\\
        \dr(x)-z_q(e) &= t - (z_q(e)-e)
    \end{align*}
    However, as $\dr(x)$ and $z_q(e)$ are cycles, $\dr(x)-z_q(e)\in \ker\dr$.
    In addition, $t$ and $z_q(e)-e$ are in $\Span A_q$ so $t - (z_q(e)-e)\in\Span A_q$.
    Since $\Span A_q \cap \ker\dr = \{0\}$ we obtain that $\dr(x)-z_q(e) = t - (z_q(e)-e) = 0$ so $\dr(x) = z_q(e)$.

    This directly implies $[z_q(e)] = [\dr(x)] = 0$.
    By~Theorem 4.3 in~\cite{edelsbrunner-tri_partition}, $(z_q(\epsilon))_{\epsilon\in E_q}$ is a homology basis, so $[z_q(e)]$ is not the 0 homology class unless $e = 0$.
    Thus, we get $\dr(x) = t+e = t$ so $\dr(x) \in \Span A_q \cap \ker\dr = \{0\}$.
    Thus, $\dr(x) = 0$ and $x\in\ker\dr$.
    This implies that $x\in \Span A_{q+1} \cap \ker\dr = \{0\}$ so $x=0$ and $\ker\dr_{|A^q}^{|A_{q+1}} = \{0\}$.
    
    This directly implies the injectivity of $\dr_{|A^q}^{|A_{q+1}}$ and the fact that $|A_{q+1}| \leq |A^{q}|$.

    \smallskip

    By duality (doing the proof for $x \in \ker {\dr^*}^{|A^q}_{|A_{q+1}}$), we obtain the injectivity of ${\dr^*}^{|A^q}_{|A_{q+1}}$ and the fact that $|A_{q+1}| \geq |A^{q}|$.
    This implies that $\dr_{|A^q}^{|A_{q+1}}$ is squared and injective so it is invertible.

    Hence, $(P,S,C)$ is a \HC.
    Moreover, it is perfect because we have $|C_q| = |E_q| = \beta_q$.